\documentclass{amsart}
\pdfoutput=1
\usepackage[leqno]{amsmath}
\allowdisplaybreaks
\usepackage{amssymb} 
\usepackage{amsthm}
\usepackage[utf8]{inputenc}
\usepackage[T1]{fontenc}
\usepackage[cal=cm, scr=rsfso]{mathalpha}
\usepackage{mathtools}
\usepackage{algorithm}
\usepackage{algpseudocode}
\usepackage{algorithmicx}
\usepackage[usenames,dvipsnames,svgnames,table]{xcolor}
\usepackage[acronym,toc]{glossaries}
\usepackage{xspace}
\usepackage{embrac} 
\ChangeEmph{(}[-.01em,.04em]{)}[.04em,-.05em]
\usepackage{csquotes}
\usepackage{enumitem}
\usepackage{graphicx}
\usepackage{float}
\usepackage{booktabs}
\usepackage[protrusion=true,expansion=true]{microtype}
\usepackage{multirow}
\usepackage{tikz}
\usepackage{caption}
\usepackage{subcaption}
\tikzset{node distance=2cm, auto}
\usetikzlibrary{angles,quotes,calc,positioning,shapes,matrix,arrows}
\usepackage{tikz-cd}
\usetikzlibrary{}
\usepackage{tkz-euclide}
\usepackage{appendix}
\usepackage{etoolbox}
\usepackage{url}
\usepackage[hyperindex=true, colorlinks=true, linkcolor=Maroon, citecolor=Green, urlcolor=NavyBlue, breaklinks=true,linktocpage=true,linktoc=all,pagebackref]{hyperref}
\usepackage[english,capitalise,noabbrev]{cleveref}
\crefdefaultlabelformat{#2\textup{#1}#3}
\usepackage[logical-quotes,backrefs,initials]{amsrefs}

\newcommand{\C}{\mathbb{C}}

\newcommand{\B}{\mathbb{B}}
\newcommand{\RP}{\mathbb{RP}}
\newcommand{\CP}{\mathbb{CP}}
\newcommand{\HP}{\mathbb{HP}}
\newcommand{\OP}{\mathbb{OP}}
\newcommand{\R}{\mathbb{R}}
\renewcommand{\S}{\mathbb{S}}

\newcommand{\M}{\mathcal{M}}
\newcommand{\cN}{\mathcal{N}}
\newcommand{\greek}[1]{\mathrm{#1}}
\newcommand{\Beta}{\greek{B}}
\newcommand{\contained}{\subset}
\newcommand{\containedeq}{\subseteq}

\newcommand{\intersection}{\cap}

\newcommand{\suchthat}{\, : \,}
\newcommand{\st}{\suchthat}

\newcommand{\isomorphic}{\cong}

\DeclarePairedDelimiterX\norm[1]\lVert\rVert{\ifblank{#1}{\:\cdot\:}{#1}}
\DeclarePairedDelimiterX\abs[1]\lvert\rvert{\ifblank{#1}{\:\cdot\:}{#1}}
\DeclarePairedDelimiterX\set[1]{\{}{\}}{\ifblank{#1}{\: \:}{#1}}
\DeclarePairedDelimiterX\innerprod[2]\langle\rangle{\ifblank{#1#2}{\,\cdot,\cdot\,}{#1,#2}}
\DeclarePairedDelimiterX\floor[1]\lfloor\rfloor{\ifblank{#1}{\:\cdot\:}{#1}}
\DeclarePairedDelimiterXPP\expectation[2]{\ifblank{#1}{\mathbb{E}}{\mathbb{E}_{#1}}}\lbrack\rbrack{}{\ifblank{#2}{\:\cdot\:}{#2}}
\newcommand{\vol}{\operatorname{vol}}

\newcommand{\longmap}[5]{
	\begin{array}{cccc}
		#1 \colon & #2 & \longrightarrow & #3,\\[2pt]
		&#4 & \longmapsto & #5
	\end{array}
}
\newcommand{\colvec}[1]{
	\mathchoice{
		\begin{pmatrix}
			#1
		\end{pmatrix}
	}
	{
		\begin{psmallmatrix}
			#1
		\end{psmallmatrix}
	}
	{}{}
} 
\newcommand{\fiberbundle}[4]{#2\hookrightarrow#3\overset{#1}{\to}#4}
\newcommand{\composed}{\circ}
\newcommand{\NJac}{\operatorname{NJac}}
\newcommand{\Jac}{\operatorname{Jac}}

\newcommand{\diam}{\operatorname{diam}}

\newcommand{\unif}{\operatorname{unif}}

\newcommand{\erf}{\operatorname{erf}}

\newcommand{\from}{\colon}
\renewcommand{\exp}{\mathord{\operatorname{exp}}} 

\newacronym[plural=\textup{\textsc{cross}es},firstplural=compact rank one symmetric spaces \textup{(\textsc{cross}es)}]{cross}{\textup{\textsc{cross}}}{compact rank one symmetric space}
\newacronym{simo}{\textup{\textsc{simo}}}{single-input-multiple-output}

\theoremstyle{plain}
\newtheorem{theorem}{Theorem}[section]

\newtheorem{corollary}[theorem]{Corollary}

\newtheorem{lemma}[theorem]{Lemma}
\newtheorem{proposition}[theorem]{Proposition}

\theoremstyle{definition}

\newtheorem{definition}[theorem]{Definition}

\newtheorem{example}[theorem]{Example}

\theoremstyle{remark}

\floatname{algorithm}{\normalfont\scshape Algorithm}
\makeatletter
\renewcommand\fnum@algorithm{\fname@algorithm~\thealgorithm:}
\makeatother

\begin{document}

	\title[Measure-preserving mappings]{Measure-preserving mappings from the unit cube\\ to some symmetric spaces}
	
	\author{Carlos Beltrán}
	\address{Carlos Beltrán: Departamento de Matemáticas, Estadística y Computación, Universidad de Cantabria,  Avda. Los Castros, s/n, 39005 Santander, Spain}
	\email{beltranc@unican.es}
	
	\author{Damir Ferizovi\'c}
	\address{Damir Ferizovi\'c: Department of Mathematics, KU Leuven,  Celestijnenlaan 200b, Box 2400, 3001 Leuven, Belgium}
	\email{damir.ferizovic@kuleuven.be}
	
	\author{Pedro R. López-Gómez}
	\address{Pedro R. López-Gómez: Departamento de Matemáticas, Estadística y Computación, Universidad de Cantabria,  Avda. Los Castros, s/n, 39005 Santander, Spain}
	\email{lopezpr@unican.es}
	
	\date{\today{}}
	
    \thanks{The first and third authors have been supported by grant PID2020-113887GB-I00 funded by MCIN/AEI/ 10.13039/501100011033. The second author thankfully acknowledges support by the Methusalem grant of the Flemish Government. The third author has also been supported by grant PRE2021-097772 funded by MCIN/AEI/ 10.13039/501100011033 and by ``ESF Investing in your future''.}
	
	

	\begin{abstract}
		We construct measure-preserving mappings from the $d$-dimensional unit cube to the $d$-dimensional unit ball and the compact rank one symmetric spaces, namely the $d$-dimensional sphere, the real, complex, and quaternionic projective spaces, and the Cayley plane. We also give a procedure to generate measure-preserving mappings from the $d$-dimensional unit cube to product spaces and fiber bundles under certain conditions.
	\end{abstract}
	\maketitle
	\hypersetup{linkcolor=black}
    \vspace{-0.35cm}
	\tableofcontents
	\hypersetup{linkcolor=Maroon}

\vspace{-0.35cm}
\section{Introduction and main results}

Given two measure spaces $(\Omega_1,\Sigma_1,\mu_1)$ and $(\Omega_2,\Sigma_2,\mu_2)$, a bijective mapping $\varphi\from\Omega_1\to\Omega_2$ is said to be \emph{measure preserving} if both $\varphi$ and $\varphi^{-1}$ are measurable mappings and moreover $\mu_2(A)=\mu_1(\varphi^{-1}(A))$ for every $A\in\Sigma_2$; or, equivalently, if $\mu_1(A)=\mu_2(\varphi(A))$ for every $A\in\Sigma_1$. In this work, we look for measure-preserving smooth diffeomorphisms between Riemannian manifolds. 

The problem of finding measure-preserving mappings from one manifold to another has applications in  cartography, computer graphics, medical imaging, signal processing, or, more generally, in any area that requires good discretizations of a certain space. Thus, when looking for uniform collections of points or uniform grids (that is, grids all of whose cells have the same volume) on a manifold $\M$, a frequent approach consists in generating collections or grids with that property on a simpler, easily discretizable space such as the unit cube, and then transporting them to $\M$ through a measure-preserving mapping. In this sense, most of the research has been carried out for two-dimensional and three-dimensional manifolds (see \cite{ShirleyChiu1997,HolhosRosca2014,HolhosRosca2016,HolhosRosca2019,Holhos2017,RoscaPlonka2011,RoscaMorawiecDeGraef2014} and references therein). In \cite{RoscaMorawiecDeGraef2014}, the authors also obtain a mapping from the $n$-dimensional sphere of radius $r$ in $\R^{n+1}$ to the $n$-dimensional ball of radius $R$ in $\R^n$ which generalizes the equal-area Lambert mapping.

Measure-preserving mappings are also relevant in the theory of partial differential equations on Lipschitz domains (see \cite{GriepentrogHoepnerKaiserRehberg2008}), in the generation of low-discrepancy points (see, for example, \cite{BrauchartDick2012,DeMarchiElefante2018,Alexa2022,FerizovicHofstadlerMastrianni2022, Ferizovic2022}), and, more recently, they have been used to generate projective constellations for noncoherent communications over \gls{simo} channels; see \cite{NgoDecurningeGuillaudYang2020}, where the authors construct a measure-preserving mapping from the unit square to the complex projective line $\CP^1$, or \cite{CuevasAlvarezVizosoBeltranSantamariaTucekPeters2022} for the higher dimensional case. However, to the best of our knowledge, there are no constructive procedures to generate measure-preserving mappings from the $d$-dimensional unit cube to the $d$-sphere and to the remaining projective spaces.

\subsection{Notation}

In this paper, $\lambda$ denotes the Lebesgue measure in $\R^d$, and $B^d(0,R)$ denotes the open ball of radius $R\in(0,\infty]$ in $\R^d$ (if $R=\infty$, this means just $\R^d$). When $R=1$, we denote it by $\B^d$. We will call $(0,1)^d$ the (open) unit cube. 

We denote the measure associated to the normal distribution $\mathcal N(0,c)$ in $\R^d$ by $\mu_c$, that is, 
$$
d\mu_c(x)=\frac{1}{(2\pi c)^{d/2}}e^{-\norm{x}^2/(2c)}\,d\lambda(x).
$$
It is well known that the mapping
\begin{equation}\label{eq:PhiRd}
	\longmap{\Phi_{\R^d}}{(0,1)^d}{\R^d}
	{\displaystyle\colvec{x_1\\\vdots\\x_d}}
	{\displaystyle\colvec{\sqrt{2}\erf^{-1}(2x_1-1)\\\vdots\\\sqrt{2}\erf^{-1}(2x_d-1)},}
\end{equation}
is measure preserving from $((0,1)^d,\lambda)$ to $(\R^d,\mu_{c=1})$, where $\erf^{-1}$ is the inverse of the error function $\erf\from \R\to (-1,1)$ given by
\begin{equation*}
    \erf(t)=\frac{2}{\sqrt{\pi}}\int_0^te^{-s^2}\,ds.
\end{equation*}
Given any continuous function $\omega\from(0,R)\to(0,\infty)$, we consider the associated measure in $B^d(0,R)$ given by the weight function $\omega(\norm{x})$ and denote it by $\mu_\omega$, that is,
$$
d\mu_{\omega}(x)=\omega(\norm{x})\,d\lambda(x).
$$
We will always assume that $\mu_\omega$ is a probability measure, i.e.,
\begin{equation}\label{eq:peso1}
1=\int_{x\in B^d(0,R)}d\mu_\omega(x)=\int_{x\in B^d(0,R)}\omega(\norm{x})\,d\lambda(x)=\frac{2\pi^{d/2}}{\Gamma(\frac{d}{2})}\int_0^R\omega(s)s^{d-1}\,ds.
\end{equation}
Finally, if we have a Riemannian manifold $\M$ (including the case that $\M$ is the unit cube with the standard structure, the usual $d$-sphere, or any open set of a compact Riemannian manifold), we denote by $\mathrm{unif}$ the uniform measure in $\M$ according to its volume form. For example, $(\B^d,\unif)$ is the unit ball endowed with the Lebesgue measure normalized to have volume $1$, which can be denoted in our previous notation by $\mu_{\omega=1/\vol(\B^d)}$.

\subsection{The compact rank one symmetric spaces}

The \glspl{cross} are the $n$-sphere $\S^n$ and the real, complex, quaternionic, and octonionic projective spaces $\RP^n, \CP^n$, $\HP^n$, and $\mathbb{OP}^2$. These spaces, which were classified by É. Cartan, are  examples of locally harmonic Blaschke manifolds;  in fact, Lichnerowicz's conjecture claims that the \glspl{cross} are the only Riemannian manifolds of this kind. They are also the only compact connected two-point homogeneous Riemannian manifolds. See \cite{Besse1978} for more information about these spaces. 

Let $\M$ be a \gls{cross} and let $d$, $D$, and $V$ be, respectively, its real dimension, its diameter (that is, the maximum Riemannian distance between two points in $\M$), and its volume. The exponential map based on the north pole
\begin{equation*}
	\longmap{\exp_{\M}}{B^d(0,D)}{\M}
	{v}
	{\exp_{(0,\dotsc,0,1)}(v),}
\end{equation*}
is a diffeomorphism onto $\M\setminus \mathcal{X}$, where $\mathcal{X}$ is a measure zero set (just a point in the case $\M=\S^n$ and a hyperplane for the projective spaces). Moreover, the Jacobian of $\exp_{\M}$ is known as the volume density and has the form $\Omega(\norm{v})$ for a certain function $\Omega$. As a consequence, we have the following lemma:

\begin{lemma}\label{lemma:exp_preserving}
Let $\M$ be a \gls{cross}. Then, the exponential map $\exp_{\M}$ is a measure-preserving mapping from  $(B^d(0,D),\mu_{\omega=\Omega/V})$ to $(\M,\unif)$.
\end{lemma} 

\begin{proof}
    Since $\exp_\M$ is a smooth diffeomorphism, both $\exp_{\M}$ and its inverse are measurable mappings. Now, let $A\subseteq(0,1)^d$ be a measurable set. Applying the change of variables theorem to $\exp_{\M}$ we readily get the result.
\end{proof}

\Cref{table:dov} summarizes the dimension, the diameter, the volume, the exponential map, and the volume density of these classical spaces.

\setlength{\tabcolsep}{10pt}
\begin{table}[htbp]
\begin{center}
    \caption{The volume density in the \glspl{cross} is $\omega_p(q)=\Omega(r)$, where $r=d_R(p,q)$ is the Riemannian distance. In this table, we show $r^{d-1}\Omega(r)$, where $d=\dim_{\R}(\M)$, for the \glspl{cross}. We also include the diameter $D=\diam(\M)$, the volume $V=\vol(\M)$, and the exponential map $\exp_{\M}$. Table taken from \cite{BeltranDelaTorreLizarte2022}.}
    \label{table:dov}
	\begin{tabular}{cccccl}
		\toprule
		$\M$ & $d$ & $D$ & $V$ & $\exp_{\M}(v)$ & $r^{d-1}\Omega(r)$  \\
		\midrule
		$\mathbb{S}^n$ &$n$  &	$\pi$ & $\displaystyle\frac{2\pi^{(n+1)/2}}{\Gamma(\frac{n+1}{2})}$ & $\displaystyle\colvec{\frac{v}{\norm{v}}\sin\norm{v}\\ \cos\norm{v}}$ & $\sin^{n-1}r$\\  \midrule
		$\mathbb{R}\mathbb{P}^n$&$n$ &	$\pi/2$	& $\displaystyle\frac{\pi^{(n+1)/2}}{\Gamma(\frac{n+1}{2})}$  & $\displaystyle\colvec{\frac{v}{\norm{v}}\tan\norm{v}\\ 1}$& $\sin^{n-1}r$ \\ \midrule
		$\mathbb{C}\mathbb{P}^{n}$&$2n$ &	$\pi/2$ & $\displaystyle\frac{\pi^n}{n!}$ & $\displaystyle\colvec{\frac{v}{\norm{v}}\tan\norm{v}\\ 1}$ & ${{\sin^{2n-1}r\cos r}}$\\ \midrule
		$\mathbb{H}\mathbb{P}^{n}$ & $4n$ &	$\pi/2$ & $\displaystyle\frac{\pi^{2n}}{(2n+1)!}$ & $\displaystyle\colvec{\frac{v}{\norm{v}}\tan\norm{v}\\ 1}$&${{\sin^{4n-1}r\cos^3 r}}$\\ \midrule			$\mathbb{O}\mathbb{P}^{2}$ &$16$ &	$\pi/2$ & $\displaystyle\frac{\pi^8}{1320\,\Gamma(8)}$ &$\displaystyle\colvec{\frac{v}{\norm{v}}\tan\norm{v}\\ 1}$ & ${{\sin^{15}r\cos^7 r}}$\\ \bottomrule
	\end{tabular}
\end{center}
\end{table}

\subsection{Main results}
Let $\gamma$ denote the incomplete gamma function:
\[
\gamma(t,x)=\int_0^xs^{t-1}e^{-s}\,ds.
\]
Our first main result is the following proposition, which yields a measure-preserving mapping from the unit cube to the unit ball:

\begin{proposition}\label{prop:mainresult_ball}
	Let $\varphi_{\B^d}\from (\R^d,\mu_{c=1})\to (\B^d,\unif)$ be the mapping given by
	\begin{equation*}
		\varphi_{\B^d}(x)=\frac{x}{\norm{x}}\bigg(\frac{\gamma\big(\frac{d}{2},\frac{\norm{x}^2}{2}\big)}{\Gamma(\frac{d}{2})}\bigg)^{1/d}.
	\end{equation*}
	Then, the mapping $\Phi_{\B^d}=\varphi_{\B^d}\composed\Phi_{\R^d}\from ((0,1)^d,\unif)\to (\B^d,\unif)$ is measure preserving.
\end{proposition}   

The next main result provides a procedure to generate measure-preserving mappings from the unit cube to each \gls{cross}.

\begin{theorem}\label{thm:mainresult_cross}
	Let $\M$ be a \gls{cross} and let $\varphi_\M\from (\R^d,\mu_{c=1})\to (B^d(0,D),\mu_{\omega=\Omega/V})$ be the mapping given by $\varphi_\M(x)=x\rho(\norm{x})/\norm{x}$, where $\rho=\rho(r)$ is the unique solution to 
	\begin{equation*}
		\int_{0}^{\rho}\omega(s)s^{d-1}\,ds=\frac{1}{2\pi^{d/2}}\gamma\bigg(\frac{d}{2},\frac{r^2}{2}\bigg).
	\end{equation*}
	Then, the mapping $\Phi_\M=\exp_\M\composed \varphi_\M\composed\Phi_{\R^d}\from ((0,1)^d,\unif)\to (\M,\unif)$ is measure preserving.
\end{theorem}

The following conmutative diagram illustrates the construction described in \cref{thm:mainresult_cross}:
\begin{equation*}
		\begin{tikzcd}[ampersand replacement=\&, row sep=1cm, column sep=1.5cm]
			((0,1)^d,\unif) \arrow[swap]{d}{\Phi_{\M}} \arrow{r}{\Phi_{\R^d}}  \& (\R^d,\mu_{c=1}) \arrow{d}{\varphi_{\M}}  \\
			(\M,\unif)	  \&  (B^d(0,D),\mu_{\omega=\Omega/V}) \arrow{l}{\mathord{\exp_{\M}}}
		\end{tikzcd}
\end{equation*}
 
It follows straightforwardly from the definition of measure-preserving mapping that, given two Riemannian manifolds $\M_1$ and $\M_2$, and two measure-preserving mappings $\Phi_{\M_1}\from (0,1)^{\dim(\M_1)}\to \M_1$ and $\Phi_{\M_2}\from (0,1)^{\dim(\M_2)}\to \M_2$, the mapping 
\begin{equation*}
	\longmap{\Phi_{\M_1\times \M_2}}{((0,1)^{\dim(\M_1)+\dim(\M_2)},\unif)}{(\M_1\times \M_2,\unif)}{(x,y)}{(\Phi_{\M_1}(x),\Phi_{\M_2}(y)),}
\end{equation*} 
where $x\in (0,1)^{\dim(\M_1)}$ and $y\in(0,1)^{\dim(\M_2)}$, is also measure preserving. As a consequence, since by \cref{thm:mainresult_cross} we have measure-preserving mappings from the unit cube to any \gls{cross}, we also have a constructive procedure to generate measure-preserving mappings from the unit cube to any finite product of \glspl{cross}. In this work we generalize this property to the case of fiber bundles:

\begin{theorem}\label{thm:mainresult_bundles}
	Let $E$, $B$, and $F$ be Riemannian manifolds, where we assume that the measures in $E$, $B$, and $F$ are normalized to have unit volume, and let $\fiberbundle{\pi}{F}{E}{B}$ be a smooth fiber bundle such that $\NJac\pi(x)$ is constant for every $x\in E$. Let $\Phi_{B}\from ((0,1)^{\dim(B)},\unif)\to (B,\unif)$ and $\Phi_{F}\from ((0,1)^{\dim(F)},\unif)\to (F,\unif)$ be measure-preserving mappings. Let $\Psi_y\from (F,\unif)\to (\pi^{-1}(y),\unif)$  be a measure-preserving mapping for every $y\in B$ such that the mapping $\xi\from(B\times F,\unif)\to (E,\unif)$ given by $\xi(y,z)=\Psi_y(z)$ is measurable. Then, $\xi$ is measure preserving and hence the mapping
	\begin{equation*}
		\longmap{\Phi_E}{((0,1)^{\dim(E)},\unif)}{(E,\unif)}{(y,z)}{\Psi_{\Phi_B(y)}(\Phi_F(z)),}
	\end{equation*}
	where $y\in(0,1)^{\dim(B)}$ and $z\in(0,1)^{\dim(F)}$, is measure preserving.
\end{theorem}

\subsection{Structure of the paper}

In \cref{sec:RdtoBd_general}, we prove our main technical result, which yields a procedure to generate measure-preserving mappings from $(\R^d,\mu_c)$ to $(B^d(0,R),\mu_\omega)$, and we prove \cref{prop:mainresult_ball}. In \cref{sec:cubetocross}, we prove \cref{thm:mainresult_cross} and we construct measure-preserving mappings from the unit cube to each \gls{cross}. In \cref{sec:fiberbundles}, we prove \cref{thm:mainresult_bundles} and we show an alternative procedure to construct measure-preserving mappings from the unit cube to odd-dimensional spheres using the Hopf fibration. \Cref{appendix:coarea} is devoted to the smooth coarea formula, a technical tool. Finally, in \cref{appendix} we present some auxiliary computations.

\section{A technical result}\label{sec:RdtoBd_general}

Recall that $\omega\from(0,R)\to(0,\infty)$ is any continuous function satisfying \eqref{eq:peso1}.

\begin{theorem}\label{th:maintech}
Let $G(\rho)=\int_0^\rho\omega(s)s^{d-1}\,ds$, and let $\rho=\rho(r)$ be the unique solution to
\begin{equation}\label{eq:rho}
G(\rho)=\frac{1}{2\pi^{d/2}}\gamma\bigg(\frac{d}{2},\frac{r^2}{2c}\bigg).
\end{equation}
Then, the mapping $\varphi\from(\R^d,\mu_c)\to(B^d(0,R),\mu_\omega)$ given by $\varphi(x)=x\rho(\norm{x})/\norm{x}$ is measure preserving.
\end{theorem}

\begin{proof}
First, note that $\omega(\rho)>0$ implies that $G$ is an increasing function with $G(0)=0$. Moreover, \cref{eq:peso1} implies that $G(R)=\Gamma(d/2)/2\pi^{d/2}$, which means that $\rho$ is a well-defined bijection with $\rho(r\to\infty)\to R$. The inverse of $\varphi$ is easily computed: $\varphi^{-1}(y)=y\rho^{-1}(\norm{y})/\norm{y}$.

Computing the derivative with respect to $r$ at both sides of \eqref{eq:rho}, we get
\begin{equation}\label{eq:aux1}
    \omega(\rho)\rho^{d-1}\rho'(r)=\frac{r^{d-1}}{(2\pi c)^{d/2}}e^{-r^2/(2c)}.
\end{equation}
Now, let $f(r)=\rho(r)/r$ and compute the Jacobian of $\varphi(x)=xf(\norm{x})$ by choosing an orthonormal basis $v_1^x,\ldots,v_d^x$ at $x\in\R^d$, with $v_1^x=x/\norm{x}$. A straightforward computation shows that $D\varphi(x)$ preserves the orthogonality of that basis and yields
\begin{align*}
\Jac\varphi(x)=&f(\norm{x})^{d-1}(f(\norm{x})+\norm{x}f'(\norm{x}))\\
=&\frac{\rho(\norm{x})^{d-1}}{\norm{x}^{d-1}}\rho'(\norm{x})\\
\stackrel{\eqref{eq:aux1}}{=}&\frac{1}{(2\pi c)^{d/2}\omega(\rho(\norm{x}))}e^{-\norm{x}^2/(2c)}.
\end{align*}
Then, given any measurable set $\mathcal A\subseteq \R^d$, we can check that the measure of $\mathcal A$ in $(\R^d,\mu_c)$ equals that of $\varphi(\mathcal A)$ in $(B^d(0,R),\mu_\omega)$ using the change of variables theorem: if we denote by $\chi_{\mathcal A}$ the characteristic function of $\mathcal A$, then
\begin{align*}
   \mu_c(\mathcal A)=& \int_{x\in\R^d}\chi_{\mathcal A}(x)\,d\mu_c(x)\\=&\int_{x\in\R^d}\chi_{\mathcal A}(x)\frac{1}{(2\pi c)^{d/2}}e^{-\norm{x}^2/(2c)}\,d\lambda(x)\\
    =&\int_{x\in\R^d}\chi_{\mathcal A}(x)\omega(\rho(\norm{x}))\Jac\varphi(x)\,d\lambda(x)\\
    =&\int_{y\in B^d(0,R)}\chi_{\mathcal A}(\varphi^{-1}(y))\omega(\rho(\norm{\varphi^{-1}(y)}))\,d\lambda(y)\\
    =&\int_{y\in B^d(0,R)}\chi_{\varphi(\mathcal A)}(y)\omega(\norm{y})\,d\lambda(y)
    \\
    =&\int_{y\in B^d(0,R)}\chi_{\varphi(\mathcal A)}(y)\,d\mu_\omega(y)\\=&\mu_\omega(\varphi(\mathcal A)),
\end{align*}
which proves the theorem.
\end{proof}

\begin{example}[A measure-preserving mapping from $(\R^d,\mu_c)$ to $(\R^d,\mu_b)$]\label{example:RdtoRd_normals}
	In this case, we have $R=\infty$ and 
	\begin{equation*}
		\omega(\rho)=\frac{e^{-\rho^2/(2b)}}{(2\pi b)^{d/2}}.
	\end{equation*}
	Following \cref{th:maintech}, let
	\begin{equation*}
		G(\rho)=\int_0^\rho \omega(s)s^{d-1}\,ds=\frac{1}{(2\pi b)^{d/2}} \int_0^\rho s^{d-1}e^{-s^2/(2b)}\,ds=\frac{1}{2\pi^{d/2}}\gamma\bigg(\frac{d}{2},\frac{\rho^2}{2b}\bigg).
	\end{equation*}
	We have to obtain $\rho$ from
	\begin{equation*}
		\frac{1}{2\pi^{d/2}}\gamma\bigg(\frac{d}{2},\frac{\rho^2}{2b}\bigg)=\frac{1}{2\pi^{d/2}}\gamma\bigg(\frac{d}{2},\frac{r^2}{2c}\bigg),
	\end{equation*}
	which is obviously solved by $\rho(r)=r\sqrt{b/c}$. 
	Thus, we conclude that
	\begin{equation*}
		\varphi(x)=x\sqrt{\frac{b}{c}}
	\end{equation*}
	defines a measure-preserving mapping from $(\R^d,\mu_c)$ to $(\R^d,\mu_b)$.
\end{example}

\begin{example}[A measure-preserving mapping from $(\R^d,\mu_{c=1}))$ to $(\R^d,\mu_{\text{stereo}})$]\label{example:RdtoRd_mu}
	Let $\mu_{\text{stereo}}$ be the measure that makes the stereographic projection a measure-preserving mapping, that is,
	\begin{equation*}
		d\mu_{\text{stereo}}(x)=\frac{1}{\vol(\S^d)}\frac{2^d}{(1+\norm{x}^2)^d}\,d\lambda(x).
	\end{equation*}
	In this case, we have $c=1$, $R=\infty$, and
	\begin{equation*}
		\omega(\rho)=\frac{1}{\vol(\S^d)}\frac{2^d}{(1+\rho^2)^d}=\frac{\Gamma(\frac{d+1}{2})}{2\pi^{(d+1)/2}}\frac{2^d}{(1+\rho^2)^d}.
	\end{equation*}
	We compute
	\begin{equation*}
		G(\rho)=\int_0^\rho \omega(s)s^{d-1}\,ds=\frac{2^{d-1}\Gamma(\frac{d+1}{2})}{\pi^{(d+1)/2}}\int_0^\rho \frac{s^{d-1}}{(1+s^2)^d}\,ds.
	\end{equation*}
	If $d=2$, 
	\begin{equation*}
		G(\rho)=\frac{2\Gamma(\frac{3}{2})}{\pi^{3/2}}\int_0^\rho \frac{s}{(1+s^2)^2}\,ds=\frac{\rho^2}{2\pi(1+\rho^2)}.
	\end{equation*}
	Therefore, we have to obtain $\rho$ from
	\begin{equation*}
		\frac{\rho^2}{2\pi(1+\rho^2)}=\frac{1}{2\pi}\gamma\bigg( 1,\frac{r^2}{2}\bigg),
	\end{equation*}
	that is,
	\begin{equation*}
		\rho^2=\frac{1}{e^{-r^2/2}}-1=e^{r^2/2}-1.
	\end{equation*}
	Hence, we conclude that
	\begin{equation*}
		\varphi(x)=\frac{x}{\norm{x}}\sqrt{e^{\norm{x}^2/2}-1}
	\end{equation*}
	defines a measure-preserving mapping from $(\R^2,\mu_{c=1})$ to $(\R^2,\mu_{\text{stereo}})$.
\end{example}

\begin{example}[A measure-preserving mapping from $(\R^d,\mu_{c=1})$ to $(\B^d,\unif)$]\label{example:Rdtoball}
	In this case, we have $c=1$, $R=1$, and
	\begin{equation*}
		\omega(\rho)=\frac{1}{\vol(\B^d)}=\frac{\Gamma(\frac{d}{2}+1)}{\pi^{d/2}}=\frac{d\,\Gamma(\frac{d}{2})}{2\pi^{d/2}}.
	\end{equation*}
	We easily compute
	\begin{equation*}
		G(\rho)=\int_0^\rho \omega(s)s^{d-1}\,ds=\frac{\Gamma(\frac{d}{2})}{2\pi^{d/2}}\rho^d,
	\end{equation*}
and we have to obtain $\rho$ from
	\begin{equation*}
		\frac{\Gamma(\frac{d}{2})}{2\pi^{d/2}}\rho^d=\frac{1}{2\pi^{d/2}}\gamma\bigg(\frac{d}{2},\frac{r^2}{2}\bigg),
	\end{equation*}
concluding that, following the notation of \cref{prop:mainresult_ball},
	\begin{equation*}
		\varphi_{\B^d}(x)=\frac{x}{\norm{x}}\Bigg(\frac{\gamma\big(\frac{d}{2},\frac{\norm{x}^2}{2}\big)}{\Gamma(\frac{d}{2})} \Bigg)^{1/d}
	\end{equation*}
	defines a measure-preserving mapping from $(\R^d,\mu_{c=1})$ to $(\B^d,\unif)$.
\end{example}

\begin{proof}[Proof of \cref{prop:mainresult_ball}]
    Immediate from \cref{example:Rdtoball} and the fact that $\Phi_{\R^d}$ is measure preserving.
\end{proof}

\section{Measure-preserving mappings from the unit cube to the compact rank one symmetric spaces}\label{sec:cubetocross}

After \cref{th:maintech}, the proof of \cref{thm:mainresult_cross} is now straightforward:

\begin{proof}[Proof of \cref{thm:mainresult_cross}]
    Immediate from \cref{th:maintech}, \cref{lemma:exp_preserving}, and the fact that $\Phi_{\R^d}$ is measure preserving.
\end{proof}

We can now generate measure-preserving mappings from the unit cube to all the \glspl{cross} following \cref{th:maintech}: it suffices to consider $\exp_{\M}\composed\varphi_\M\composed \Phi_{\R^d}$, where, according to \cref{th:maintech}, the mapping $\varphi_\M\from(\R^d,\mu_{c=1})\to(B^d(0,D),\mu_{\omega=\Omega/V})$ can be computed, to some extent, explicitly. We do the computations for the different choices of $\M$ in the next few subsections. Recall that, for each \gls{cross} $\M$, we denote its real dimension by $d$, its diameter by $D$, and its volume by $V$ (see \cref{table:dov}).

\subsection{The unit sphere $\S^n$}
In this case, we have $d=n$, $D=\pi$, and
\begin{equation*}
    \omega(r)=\frac{\Omega(r)}{V}=\frac{\Gamma(\frac{n+1}{2})}{2\pi^{(n+1)/2}}\frac{\sin^{n-1}r}{r^{n-1}}.
\end{equation*}

\begin{corollary}\label{cor:sphere}
The mapping $\varphi_{\S^n}\from(\R^n,\mu_{c=1})\to(B^n(0,\pi),\mu_{\omega=\Omega/V})$ given by $\varphi_{\S^n}(x)=x\rho(\norm{x})/\norm{x}$ is measure preserving if $\rho=\rho(r)$ satisfies
\begin{equation*}
\int_0^\rho\sin^{n-1}r\,dr=\frac{\sqrt{\pi}}{\Gamma(\frac{n+1}{2})}\gamma\bigg(\frac{n}{2},\frac{r^2}{2}\bigg).
\end{equation*}
As a consequence, the mapping $\Phi_{\S^n}=\exp_{\S^n}\composed\varphi_{\S^n}\composed\Phi_{\R^n}\from((0,1)^n,\unif)\to(\S^n,\unif)$ is measure preserving. For $n=1$ we have
\begin{equation*}
    \rho(r)=\sqrt{\pi}\gamma\bigg(\frac{1}{2},\frac{r^2}{2}\bigg)=\pi\erf\bigg(\frac{r}{\sqrt{2}}\bigg),
\end{equation*}
and so 
\begin{equation*}
    \Phi_{\S^1}(x)=(-\sin{2\pi x},-\cos{2\pi x})\isomorphic-ie^{-i2\pi x}.
\end{equation*}
For $n=2$ we can compute $\rho(r)$ explicitly:
\begin{equation*}
    \rho(r)=2\arccos e^{-r^2/4},
\end{equation*}
and hence
\begin{equation*}
    \Phi_{\S^2}(x)=\bigg(\frac{\Phi_{\R^2}(x)}{\norm{\Phi_{\R^2}(x)}}2e^{-\norm{\Phi_{\R^2}(x)}^2/4}\sqrt{1-e^{-\norm{\Phi_{\R^2}(x)}^2/2}},2e^{-\norm{\Phi_{\R^2}(x)}^2/2}-1\bigg).
\end{equation*}
\end{corollary}

\begin{proof}
From \cref{th:maintech} we just need to check that
$$
\int_0^\rho\frac{\Gamma(\frac{n+1}{2})}{2\pi^{(n+1)/2}}\sin^{n-1}r\,dr=\frac{1}{2\pi^{n/2}}\gamma\bigg(\frac{n}{2},\frac{r^2}{2}\bigg),
$$
which is equivalent to the formula in the corollary. The case $n=2$ reads
$$
\sin^2\frac{\rho}{2}=\frac{1-\cos \rho}{2}=\gamma\bigg(1,\frac{r^2}{2}\bigg)=1-e^{-r^2/2},
$$
which is equivalent to the last claim in the corollary.
\end{proof}

Note that the integral in the left-hand side of the expression in \cref{cor:sphere} is an incomplete beta function:
\begin{equation*}
    \int_0^\rho\sin^{n-1}r\,dr=2^{n-1}\Beta_{\sin^2(\rho/2)}\Big(\frac{n}{2},\frac{n}{2}\Big).
\end{equation*}
Hence, it is not possible to obtain a closed expression for $\Phi_{\S^n}$ when $n>2$. In \cref{sec:fiberbundles} we consider a different approach that provides measure-preserving mappings with closed expressions for odd-dimensional spheres.

\Cref{fig:pointsS2,fig:gridsS2} illustrate the measure-preserving mapping obtained in \cref{cor:sphere} for the particular case of the two-dimensional sphere $\S^2$.

\begin{figure}[htbp]

\begin{subfigure}[t]{0.46\textwidth}
\includegraphics[width=0.89\textwidth]{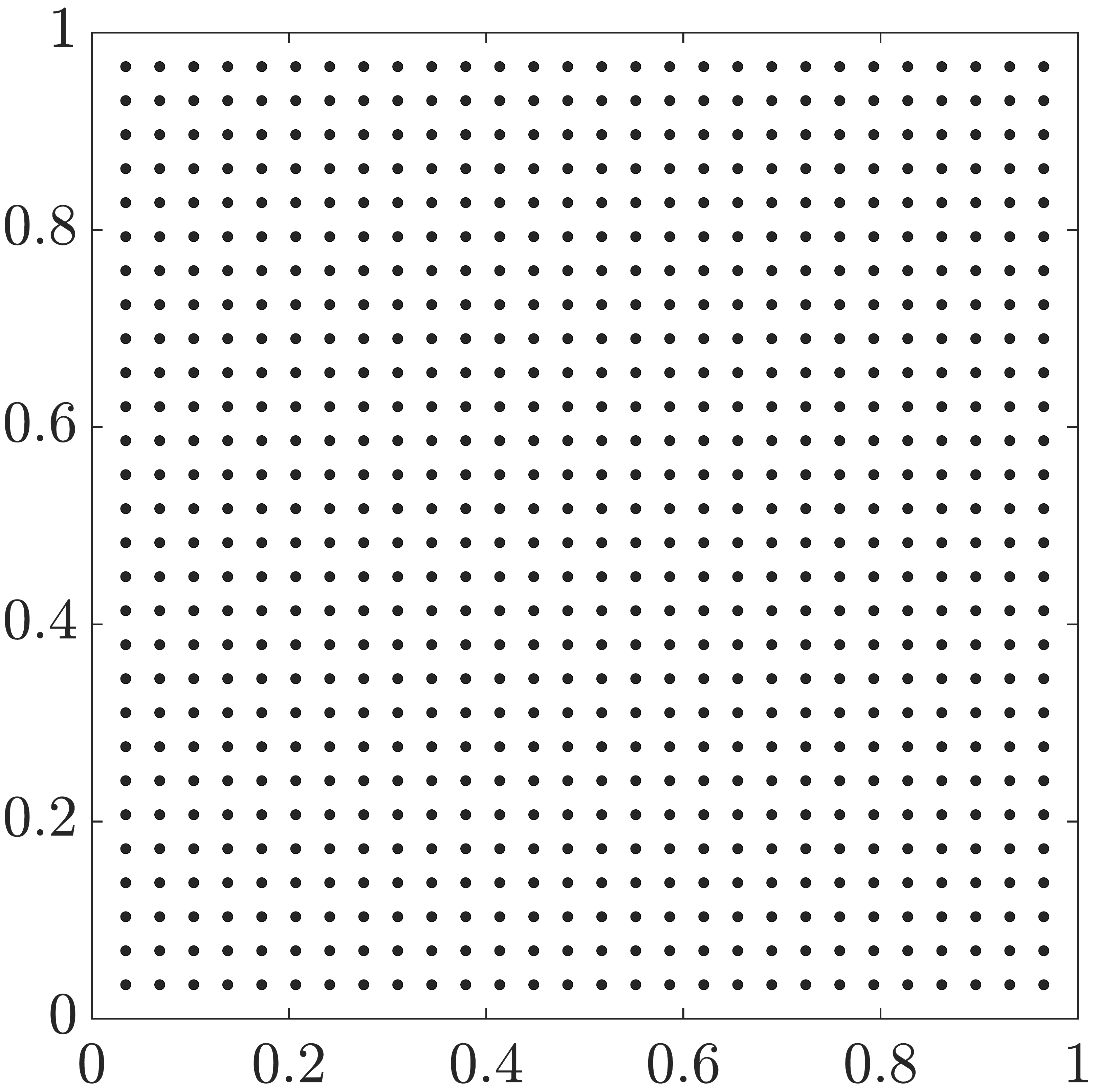}
\caption{Square mesh in $(0,1)^2$}
\label{fig:unifSquare}
\end{subfigure}
\hfill
\begin{subfigure}[t]{0.46\textwidth}
\includegraphics[width=0.89\textwidth]{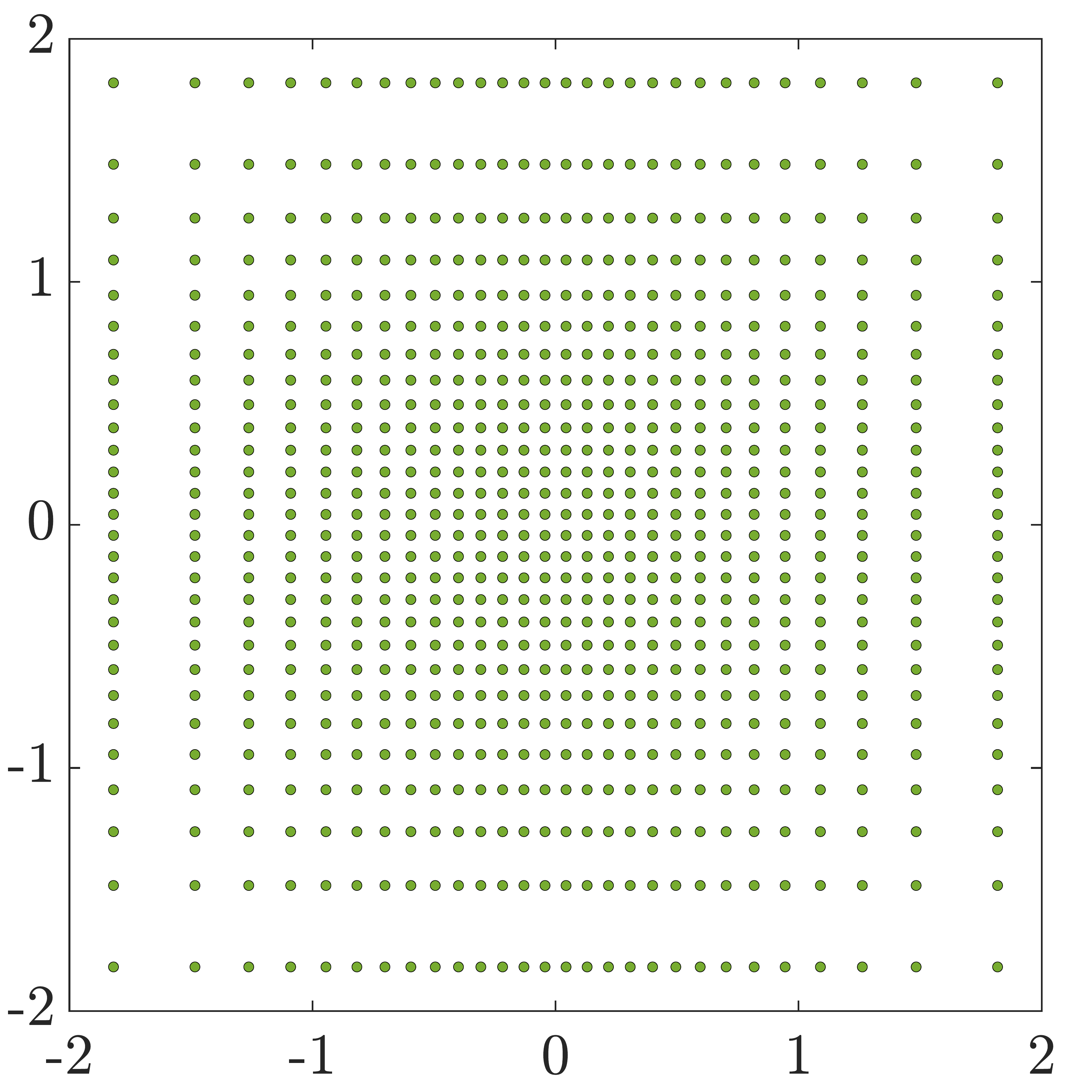}
\caption{Image of the mesh under $\Phi_{\R^2}$}
\label{fig:normalR2}
\end{subfigure}
\begin{subfigure}[t]{0.45\textwidth}
\hspace{0.2cm}\includegraphics[width=0.88\textwidth]{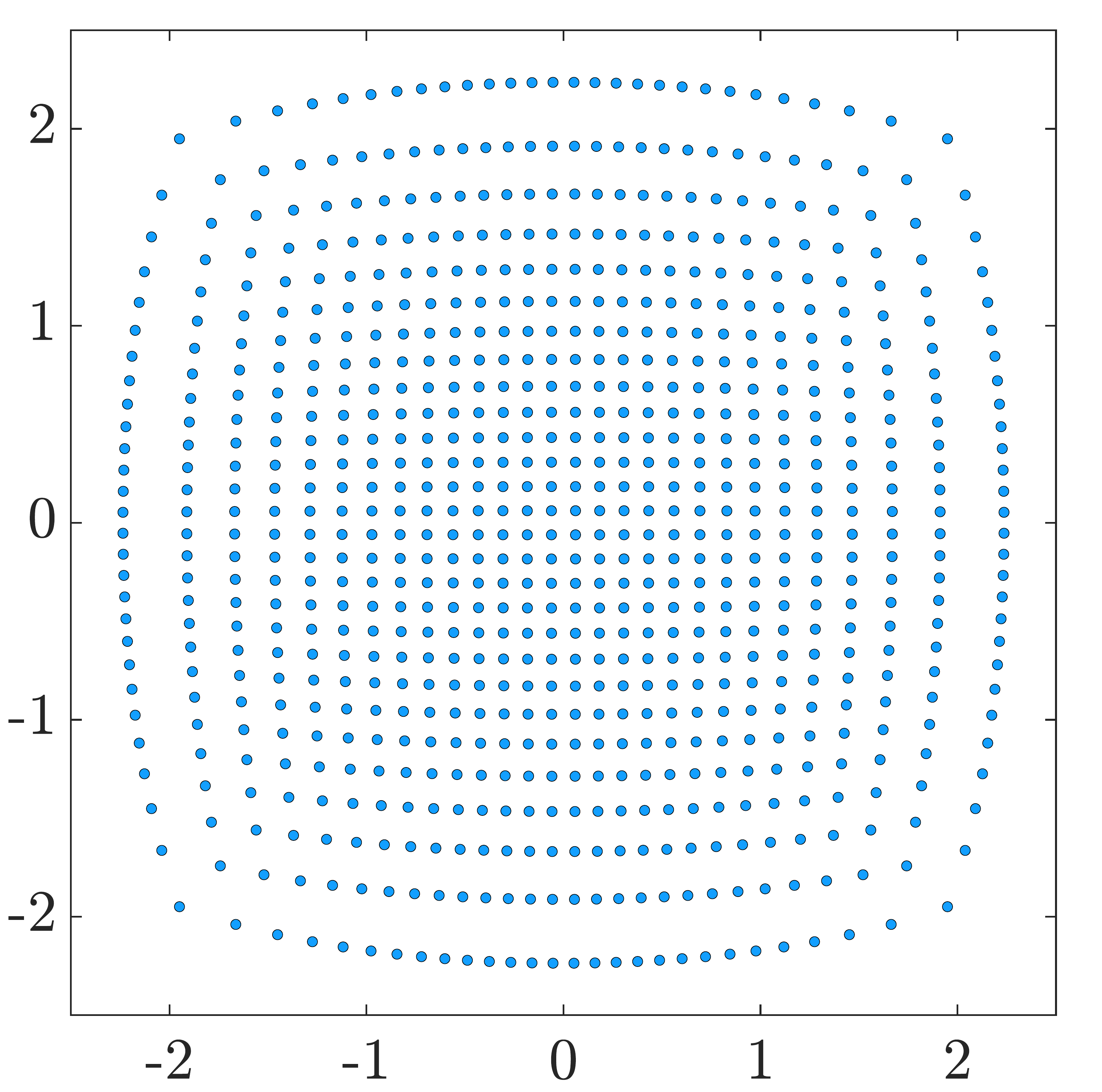}
\caption{Image of the mesh under $\varphi_{\S^2}\circ \Phi_{\R^2}$}
\label{fig:unifBall}
\end{subfigure}
\hfill
\begin{subfigure}[t]{0.47\textwidth}
\includegraphics[width=0.95\textwidth]{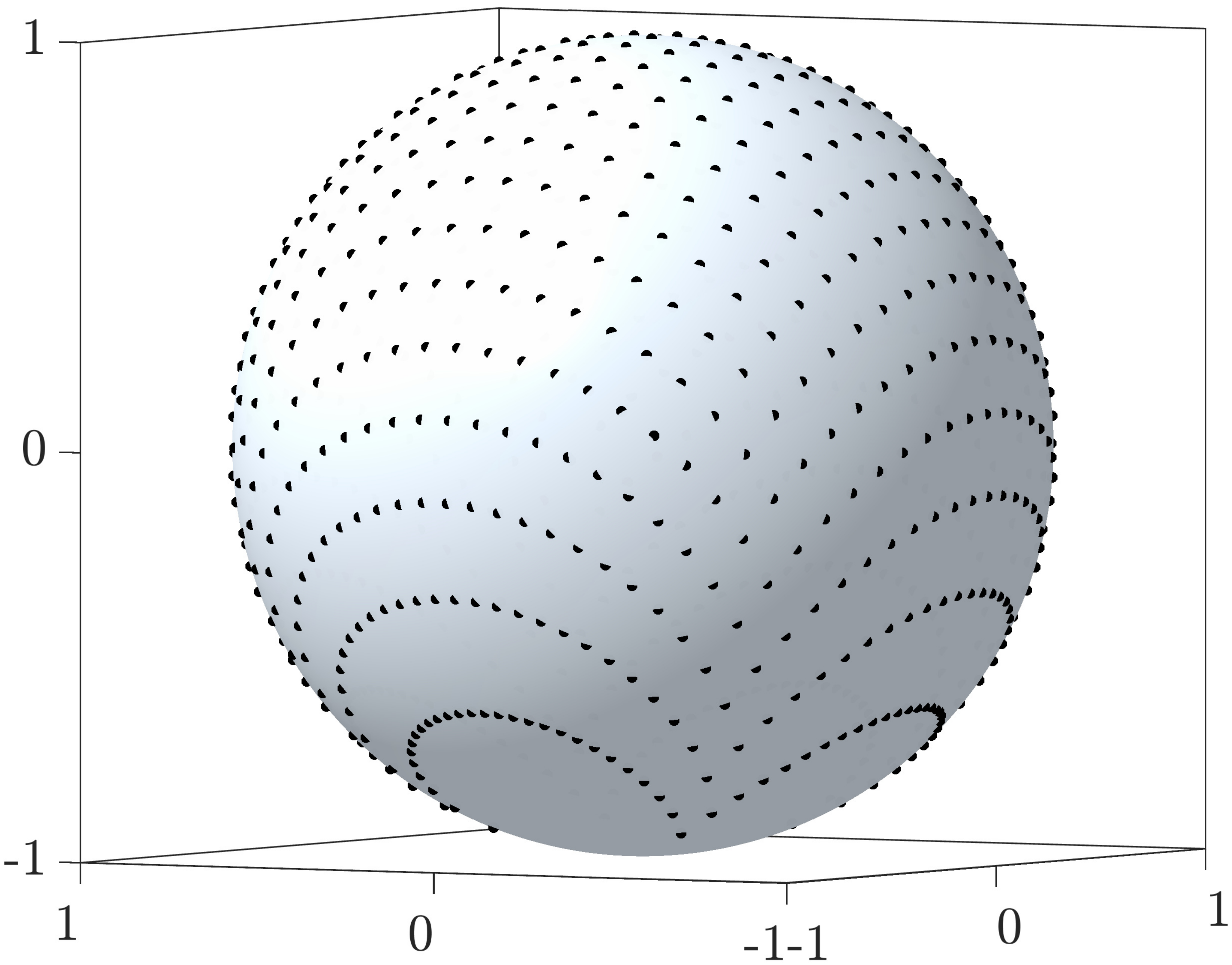}
\caption{Image of the mesh under $\mathrm{exp}_{\S^2}\circ\varphi_{\S^2}\circ\Phi_{\R^2}$, lateral view}
\label{fig:unifSphere}
\end{subfigure}
\hfill
\begin{subfigure}[t]{0.45\textwidth}
\includegraphics[width=0.89\textwidth]{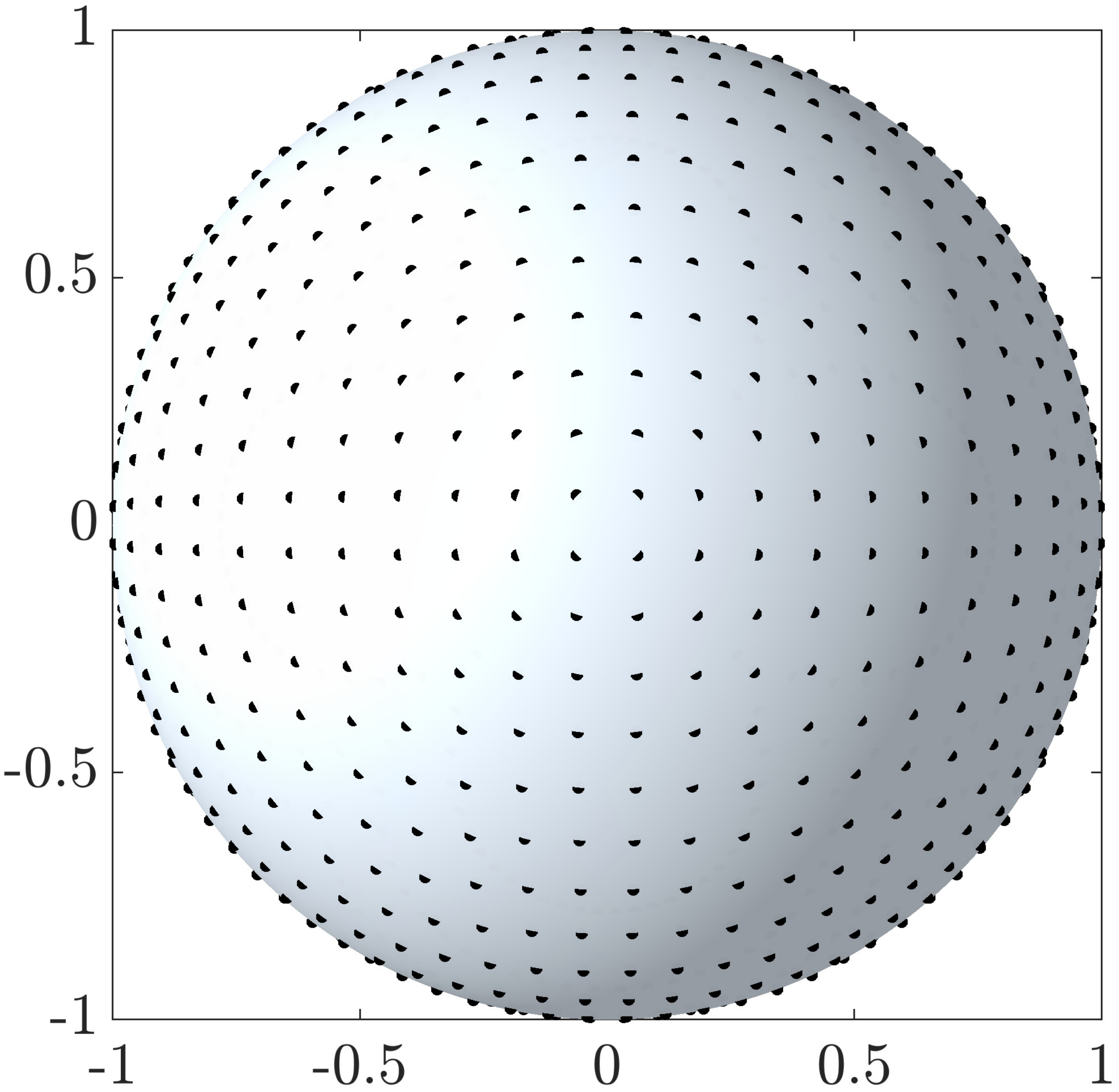}
\caption{Image of the mesh under $\mathrm{exp}_{\S^2}\circ\varphi_{\S^2}\circ\Phi_{\R^2}$, north pole view}
\label{fig:unifSphereNorth}
\end{subfigure}
\hfill
\begin{subfigure}[t]{0.45\textwidth}
\includegraphics[width=0.89\textwidth]{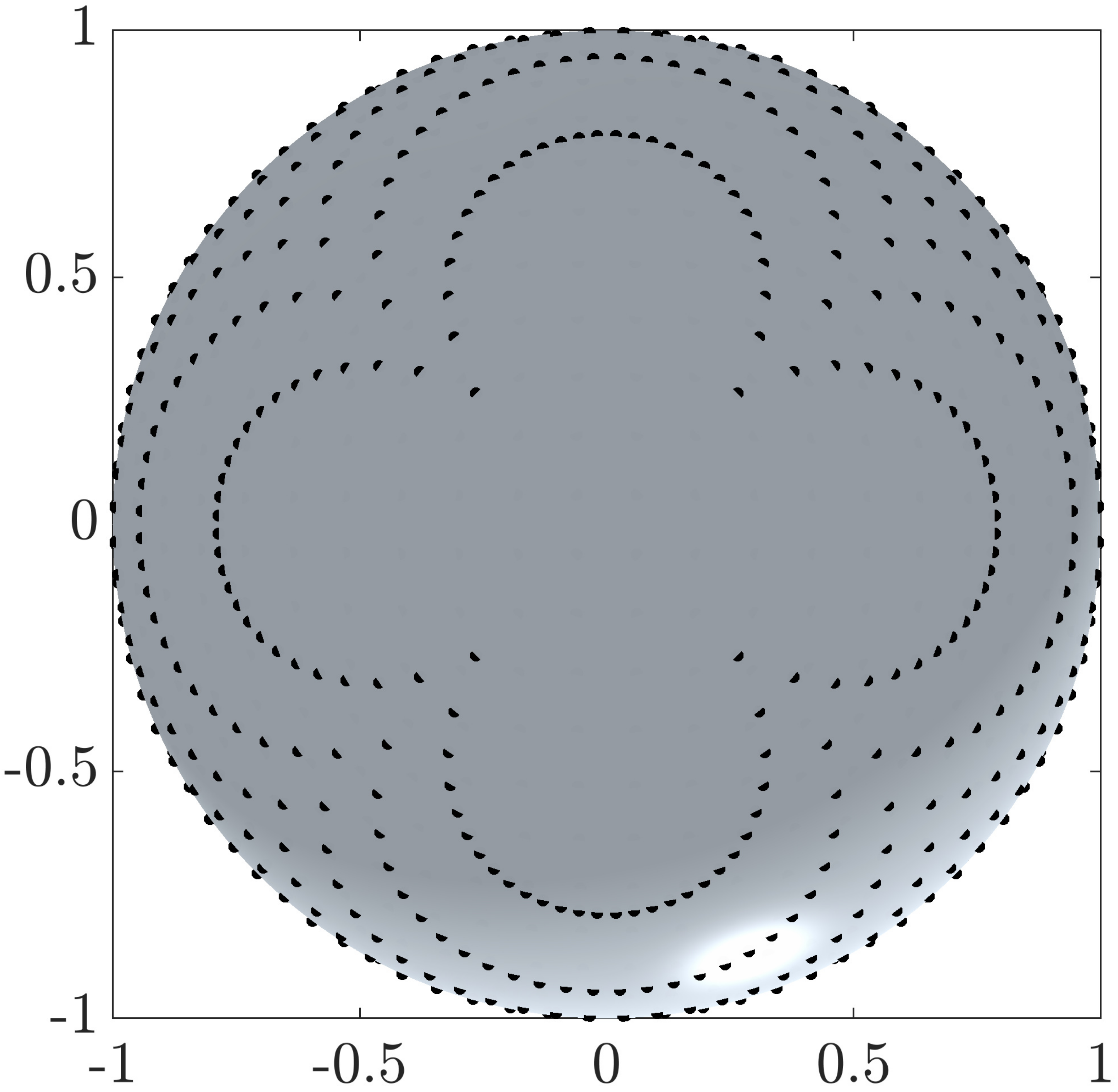}
\caption{Image of the mesh under $\mathrm{exp}_{\S^2}\circ\varphi_{\S^2}\circ\Phi_{\R^2}$, south pole view}
\label{fig:unifSphereSouth}
\end{subfigure}

\caption{The measure-preserving mapping $\Phi_{\S^2}=\exp_{\S^2}\composed\varphi_{\S^2}\composed\Phi_{\R^2}\from((0,1)^2,\unif)\to(\S^2,\unif)$ transforms points on $(0,1)^2$ into points on $\S^2$. For an initial collection of $784$ uniformly distributed points on $(0,1)^2$, we show the different steps from the unit square to the sphere.}
\label{fig:pointsS2}
\end{figure}

\begin{figure}[htbp]
\begin{subfigure}[t]{0.44\textwidth}
\raggedleft
\hspace{-0.6cm}
\includegraphics[width=\textwidth]{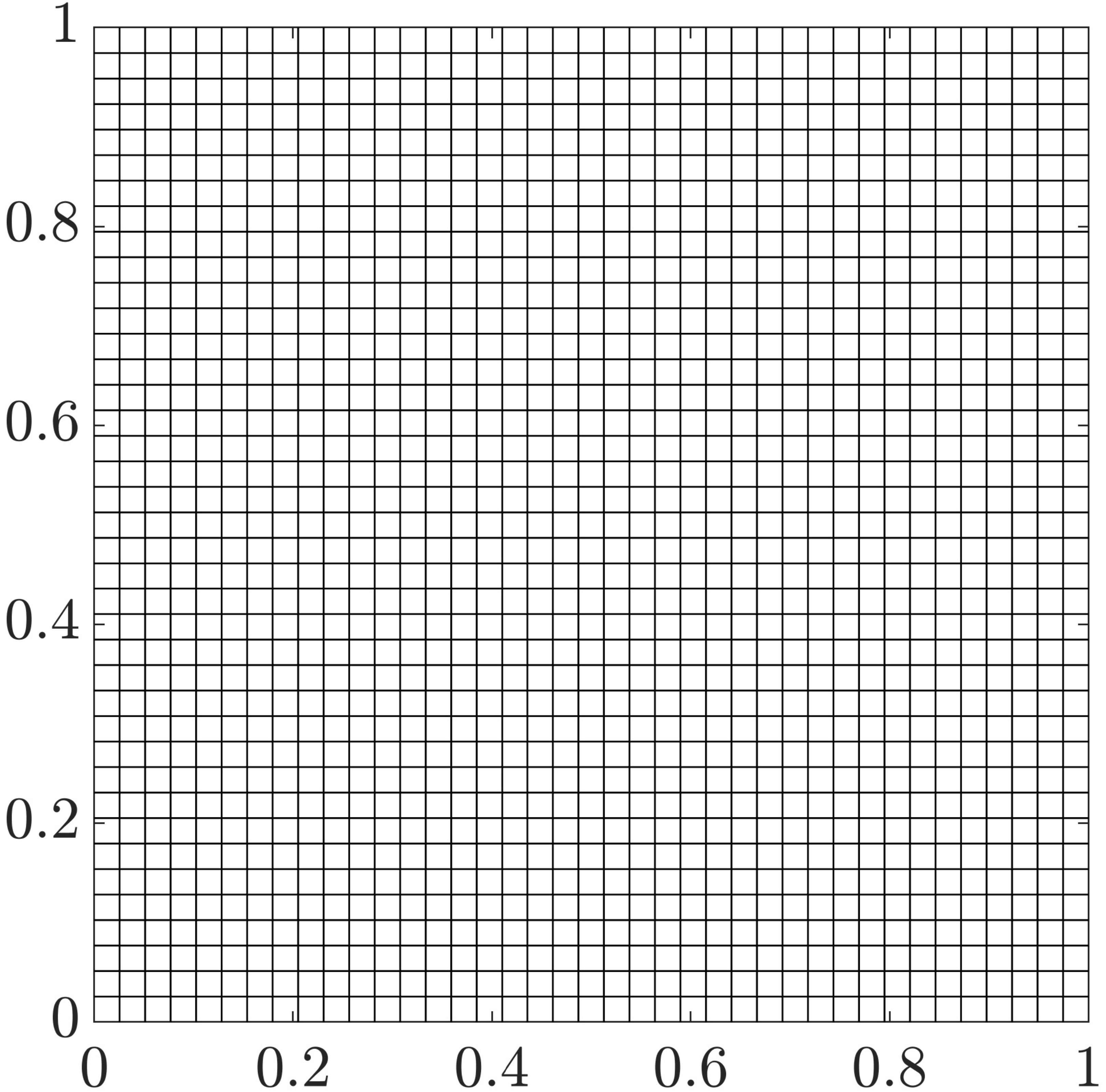}
\caption{Uniform grid in $(0,1)^2$}
\label{fig:GridSquare}
\end{subfigure}
\hfill
\begin{subfigure}[t]{0.5\textwidth}
\includegraphics[width=\textwidth]{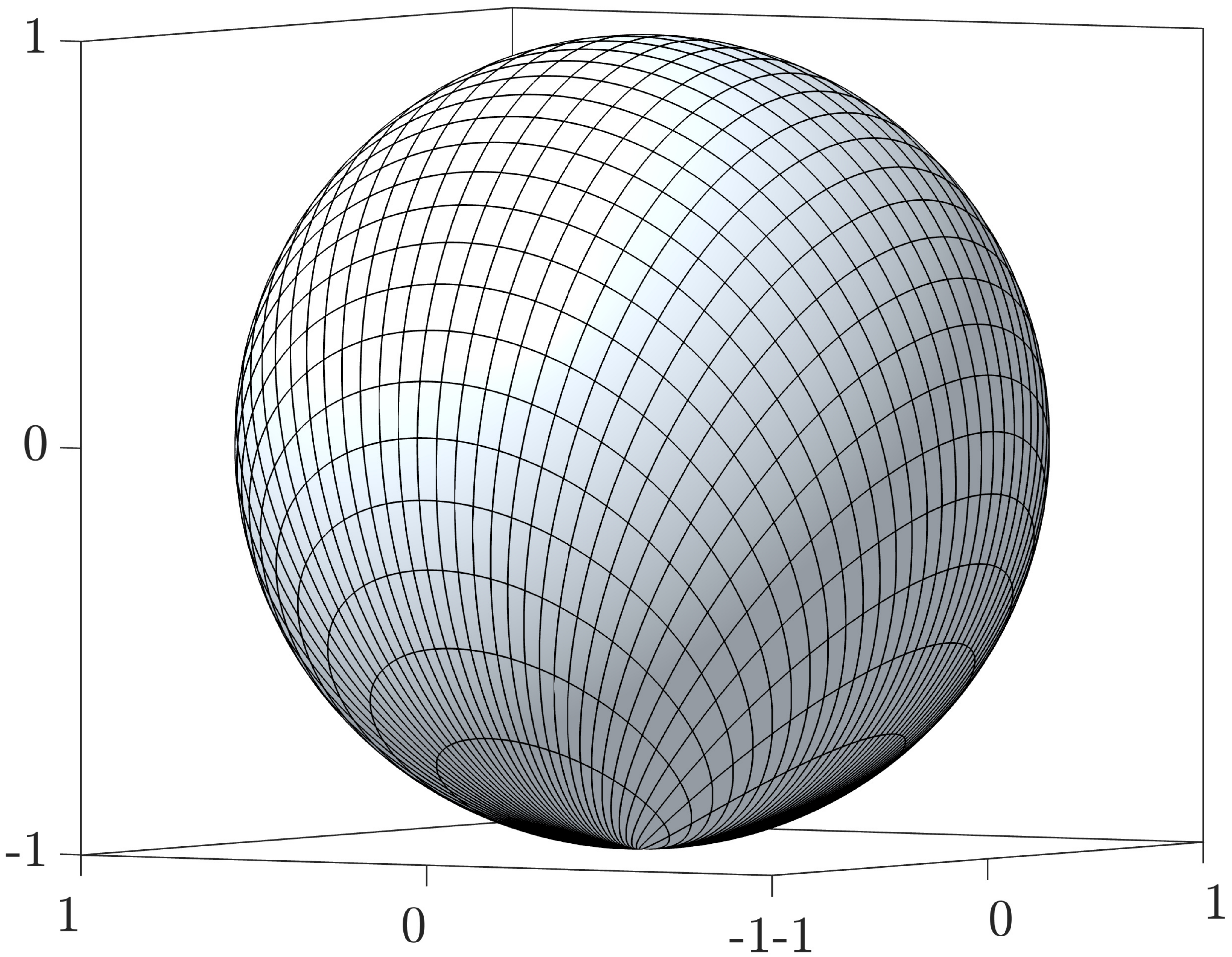}
\caption{Image of the grid in $\S^2$}
\label{fig:GridSphere}
\end{subfigure}
\hfill
\begin{subfigure}[t]{0.45\textwidth}
\includegraphics[width=\textwidth]{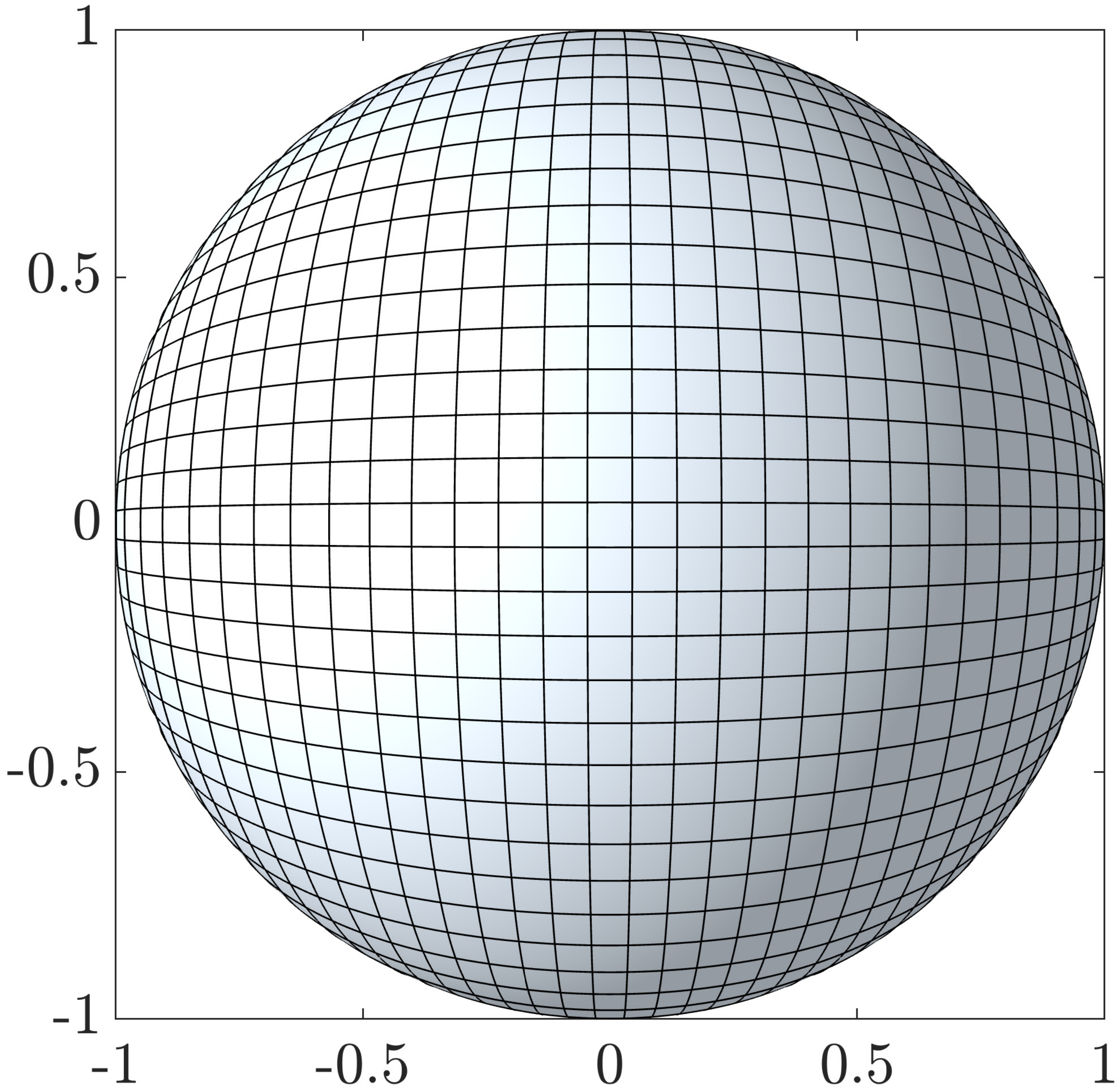}
\caption{Image of the grid in $\S^2$, north pole view}
\label{fig:GridSphereNorth}
\end{subfigure}
\hfill
\begin{subfigure}[t]{0.45\textwidth}
\includegraphics[width=\textwidth]{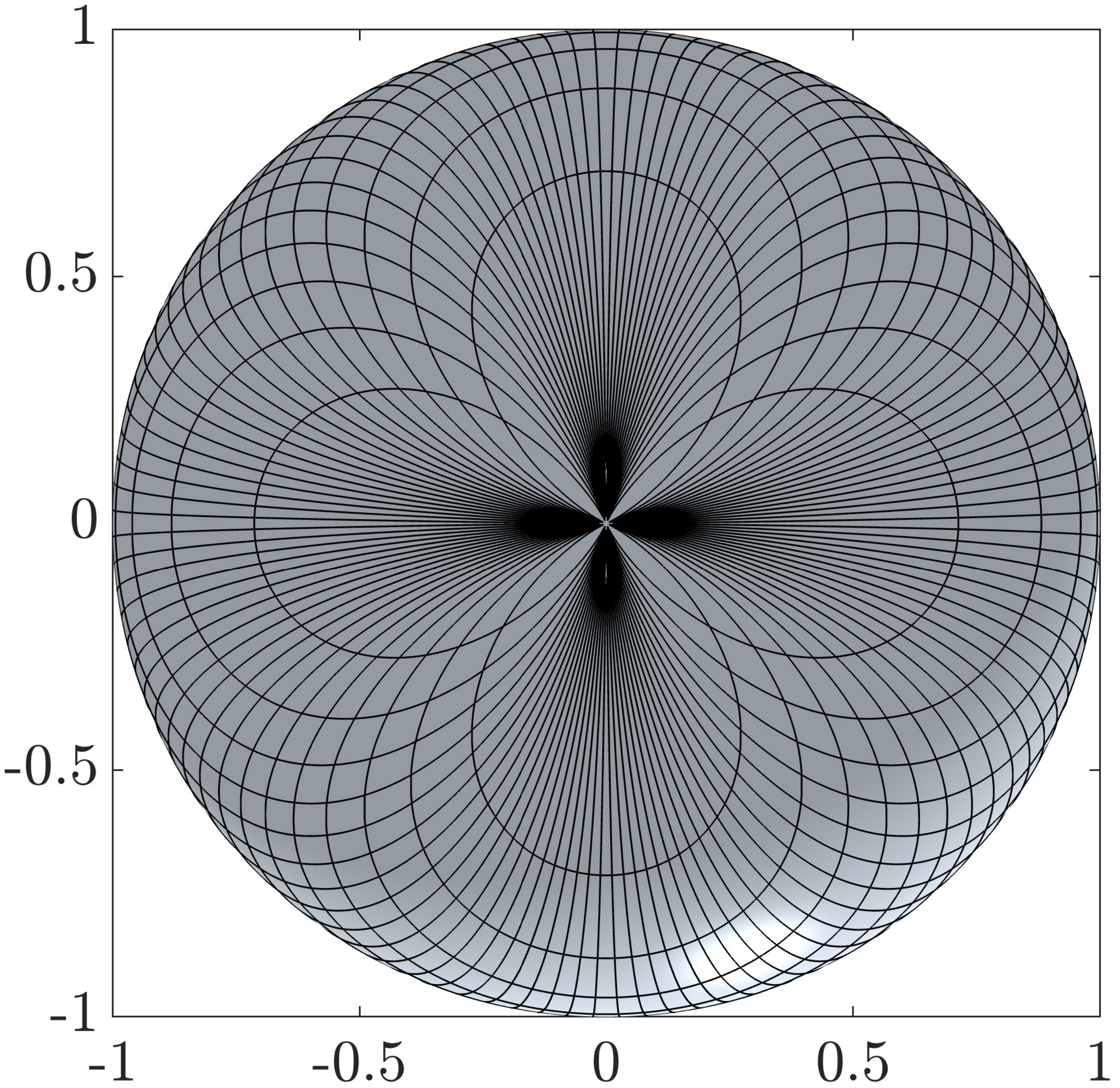}
\caption{Image of the grid in $\S^2$, south pole view}
\label{fig:GridSphereSouth}
\end{subfigure}

\caption{The measure-preserving mapping $\Phi_{\S^2}=\exp_{\S^2}\composed\varphi_{\S^2}\composed\Phi_{\R^2}\from((0,1)^2,\unif)\to(\S^2,\unif)$ transforms uniform grids in $(0,1)^2$ into uniform grids in $\S^2$. We show the image of a grid in $(0,1)^2$ formed by 1369 cells.} 
\label{fig:gridsS2}
\end{figure}

\subsection{The real projective space $\mathbb{RP}^n$}

In this case, we have $d=n$, $D=\pi/2$, and
\begin{equation*}
	\omega(r)=\frac{\Omega(r)}{V}=\frac{\Gamma(\frac{n+1}{2})}{\pi^{(n+1)/2}}\frac{\sin^{n-1}r}{r^{n-1}}.
\end{equation*}

\begin{corollary}\label{cor:RPn}
	The mapping $\varphi_{\RP^n}\from(\R^n,\mu_{c=1})\to(B^n(0,\pi/2),\mu_{\omega=\Omega/V})$ given by $\varphi_{\RP^n}(x)=x\rho(\norm{x})/\norm{x}$ is measure preserving if $\rho=\rho(r)$ satisfies
	\begin{equation*}
		\int_0^\rho\sin^{n-1}r\,dr=\frac{\sqrt{\pi}}{2\Gamma(\frac{n+1}{2})}\gamma\bigg(\frac{n}{2},\frac{r^2}{2}\bigg).
	\end{equation*}
	As a consequence, the mapping $\Phi_{\RP^n}=\exp_{\RP^n}\composed\varphi_{\RP^n}\composed\Phi_{\R^n}\from((0,1)^n,\unif)\to(\RP^n,\unif)$ is measure preserving. For $n=1$ we have
    \begin{equation*}
        \rho(r)=\frac{\sqrt{\pi}}{2}\gamma\bigg(\frac{1}{2},\frac{r^2}{2}\bigg)=\frac{\pi}{2}\erf\bigg(\frac{r}{\sqrt{2}}\bigg),
    \end{equation*}
    and so
    \begin{equation*}
        \Phi_{\RP^1}(x)=(-\cot \pi x,1).
    \end{equation*}
    For $n=2$ we can compute $\rho(r)$ explicitly:
	\begin{equation*}
	    \rho(r)=\arccos e^{-r^2/2},
	\end{equation*}
    and hence
    \begin{equation*}
        \Phi_{\RP^2}(x)=\bigg(\frac{\Phi_{\R^2}(x)}{\norm{\Phi_{\R^2}(x)}}\sqrt{e^{\norm{\Phi_{\R^2}(x)}^2}-1},1\bigg).
    \end{equation*}
\end{corollary}

\begin{proof}
From \cref{th:maintech} we just need to check that
$$
\int_0^\rho\frac{\Gamma(\frac{n+1}{2})}{\pi^{(n+1)/2}}\sin^{n-1}r\,dr=\frac{1}{2\pi^{n/2}}\gamma\bigg(\frac{n}{2},\frac{r^2}{2}\bigg),
$$
which is equivalent to the formula in the corollary. The case $n=2$ reads
$$
1-\cos\rho=\gamma\bigg(1,\frac{r^2}{2}\bigg)=1-e^{-r^2/2},
$$
which is equivalent to the last claim in the corollary.
\end{proof}

\subsection{The complex projective space $\CP^n$}
In this case, we have $d=2n$, $D=\pi/2$, and
\begin{equation*}
	\omega(r)=\frac{\Omega(r)}{V}=\frac{n!}{\pi^n}\frac{\sin^{2n-1}{r}\cos{r}}{r^{2n-1}}.
\end{equation*}

\begin{corollary}\label{cor:CPn}
	The mapping $\varphi_{\CP^n}\from(\R^{2n},\mu_{c=1})\to(B^{2n}(0,\pi/2),\mu_{\omega=\Omega/V})$ given by $\varphi_{\CP^n}(x)=x\rho(\norm{x})/\norm{x}$ is measure preserving if $\rho=\rho(r)$ satisfies
	\begin{equation*}
		\rho(r)=\arcsin\Bigg(\bigg(\frac{1}{(n-1)!}\gamma\bigg(n,\frac{r^2}{2}\bigg)\bigg)^{1/(2n)}\Bigg).
	\end{equation*}
    As a consequence, the mapping $\Phi_{\CP^n}=\exp_{\CP^n}\composed\varphi_{\CP^n}\composed\Phi_{\R^{2n}}\from((0,1)^{2n},\unif)\to(\CP^n,\unif)$ is measure preserving.
\end{corollary}

\begin{proof}
	From \cref{th:maintech} we just need to check that
	\begin{equation*}
		\int_{0}^{\rho}\frac{n!}{\pi^n}\sin^{2n-1}{r}\cos{r}\,dr=\frac{1}{2\pi^n}\gamma\bigg(n,\frac{r^2}{2}\bigg),
	\end{equation*}
	which is equivalent to
	\begin{equation*}
		\sin^{2n}(\rho)=\frac{1}{(n-1)!}\gamma\bigg(n,\frac{r^2}{2}\bigg),
	\end{equation*}
	and the corollary follows. 
\end{proof}

\subsection{The quaternionic projective space $\HP^n$}

In this case, we have $d=4n$, $D=\pi/2$, and
\begin{equation*}
	\omega(r)=\frac{\Omega(r)}{V}=\frac{(2n+1)!}{\pi^{2n}}\frac{\sin^{4n-1}{r}\cos^3{r}}{r^{4n-1}}.
\end{equation*}

\begin{corollary}\label{cor:HPn}
	The mapping $\varphi_{\HP^n}\from(\R^{4n},\mu_{c=1})\to(B^{4n}(0,\pi/2),\mu_{\omega=\Omega/V})$ given by $\varphi_{\HP^n}(x)=x\rho(\norm{x})/\norm{x}$ is measure preserving if $\rho=\rho(r)$ satisfies
	\begin{equation*}
		\int_0^\rho\sin^{4n-1}{r}\cos^3{r}\,dr=\frac{1}{2(2n+1)!}\gamma\bigg( 2n,\frac{r^2}{2}\bigg).
	\end{equation*}
	As a consequence, the mapping $\Phi_{\HP^n}=\exp_{\HP^n}\composed\varphi_{\HP^n}\composed\Phi_{\R^{4n}}\from((0,1)^{4n},\unif)\to(\HP^n,\unif)$ is measure preserving.
\end{corollary}

\begin{proof}
From \cref{th:maintech} we just need to check that
\begin{equation*}
    \int_{0}^{\rho}\frac{(2n+1)!}{\pi^{2n}}\sin^{4n-1}r\cos^{3}r\,dr=\frac{1}{2\pi^{2n}}\gamma\bigg( 2n,\frac{r^2}{2}\bigg),
\end{equation*}
which is equivalent to the formula in the corollary.
\end{proof}

\subsection{The Cayley plane $\OP^2$}

In this case, we have $d=16$, $D=\pi/2$, and
\begin{equation*}
	\omega(r)=\frac{\Omega(r)}{V}=\frac{1320\,\Gamma(8)}{\pi^{8}}\frac{\sin^{15}{r}\cos^7{r}}{r^{15}}.
\end{equation*}

\begin{corollary}\label{cor:OPn}
	The mapping $\varphi_{\OP^2}\from(\R^{16},\mu_{c=1})\to(B^{16}(0,\pi/2),\mu_{\omega=\Omega/V})$ given by $\varphi_{\OP^2}(x)=x\rho(\norm{x})/\norm{x}$ is measure preserving if $\rho=\rho(r)$ satisfies
	\begin{equation*}
		\int_0^\rho\sin^{15}{r}\cos^7{r}\,dr=\frac{1}{2640\,\Gamma(8)}\gamma\bigg(8,\frac{r^2}{2}\bigg).
	\end{equation*}
	As a consequence, the mapping $\Phi_{\OP^2}=\exp_{\OP^2}\composed\varphi_{\OP^2}\composed\Phi_{\R^{16}}\from((0,1)^{16},\unif)\to(\OP^2,\unif)$ is measure preserving.
\end{corollary}

\begin{proof}
From \cref{th:maintech} we just need to check that
\begin{equation*}
    \int_{0}^{\rho}\frac{1320\Gamma(8)}{\pi^8}\sin^{15}r\cos^{7}r\,dr=\frac{1}{2\pi^{8}}\gamma\bigg( 8,\frac{r^2}{2}\bigg),
\end{equation*}
which is equivalent to the formula in the corollary.
\end{proof}

In \cref{table:explicit-formulas} we show the cases for which we have a closed expression for the measure-preserving mapping $\Phi_{\M}$. In addition, in the next section we present an approach that will allow us to obtain measure-preserving mappings with explicit expressions for any odd-dimensional sphere.

\setlength{\tabcolsep}{13pt}
\begin{table}[htbp]
\begin{center}
    \caption{Summary of the manifolds for which we have a closed formula for the measure-preserving mapping $\Phi_{\M}$, where $\Phi_{\R^d}$ is as in \cref{eq:PhiRd}. The computations are straightforward from our main results; see Appendix \ref{appendix}.}
    \label{table:explicit-formulas}
	\begin{tabular}{ll}
		\toprule
		$\M$ & $\Phi_{\M}=\exp_{\M}\composed\varphi_\M\composed \Phi_{\R^d}\from((0,1)^d,\unif)\to(\M,\unif)$ \\
		\midrule
        $\B^n$ & $\displaystyle\frac{\Phi_{\R^n}(x)}{\norm{\Phi_{\R^n}(x)}}\Bigg(\frac{\gamma\big(\frac{n}{2},\frac{\norm{\Phi_{\R^n}(x)}^2}{2}\big)}{\Gamma(\frac{n}{2})} \Bigg)^{1/n}$\\[0.5cm]
		$\S^1$ & $\displaystyle(-\sin 2\pi x,-\cos 2\pi x)$\\[0.5cm]
        $\S^2$ & $\displaystyle\bigg(\frac{\Phi_{\R^2}(x)}{\norm{\Phi_{\R^2}(x)}}2e^{-\norm{\Phi_{\R^2}(x)}^2/4}\sqrt{1-e^{-\norm{\Phi_{\R^2}(x)}^2/2}},2e^{-\norm{\Phi_{\R^2}(x)}^2/2}-1\bigg)$\\[0.5cm]
        $\RP^1$ & $\displaystyle(-\cot\pi x,1)$\\[0.5cm]
        $\RP^2$ & $\displaystyle\bigg(\frac{\Phi_{\R^2}(x)}{\norm{\Phi_{\R^2}(x)}}\sqrt{e^{\norm{\Phi_{\R^2}(x)}^2}-1},1\bigg)$\\[0.5cm]
        $\CP^1$ & $\displaystyle\bigg(\frac{\Phi_{\R^2}(x)}{\norm{\Phi_{\R^2}(x)}}\sqrt{e^{\norm{\Phi_{\R^2}(x)}^2/2}-1},1\bigg)$\\[0.5cm]
        $\CP^n$ & $\displaystyle\Bigg(\frac{\Phi_{\R^{2n}}(x)}{\norm{\Phi_{\R^{2n}}(x)}}\bigg(-1+\frac{1}{1-\big(\frac{1}{(n-1)!}\gamma\big(n,\frac{\norm{\Phi_{\R^{2n}}(x)}^2}{2}\big)\big)^{1/n}}\bigg)^{1/2},1\Bigg)$\\
        \bottomrule
	\end{tabular}
\end{center}
\end{table}

\section{Measure-preserving mappings from the unit cube to fiber bundles}\label{sec:fiberbundles}

In this section we show how to construct measure-preserving mappings from the unit cube to the total space $E$ of the smooth fiber bundle $\fiberbundle{\pi}{F}{E}{B}$, where the total space $E$, the base space $B$, and the fiber $F$ are Riemannian manifolds, assuming that we have measure-preserving mappings from the corresponding unit cubes to $B$ and $F$. 

To prove \cref{thm:mainresult_bundles} we need the following lemma. The main technical tool used in its proof is the smooth coarea formula, an integral formula due to Federer~\cite{Federer1969} and Howard~\cite{Howard1993} that generalizes the change of variables formula and Fubini's theorem (see \cref{appendix:coarea}). We refer the interested reader to \cite[Section 2]{Beltran2011}.

\begin{lemma}\label{lemma:xi}
	Let $E$, $B$, and $F$ be finite-volume Riemannian manifolds, and let $\fiberbundle{\pi}{F}{E}{B}$ be a smooth fiber bundle such that $\NJac\pi(x)$ is constant for every $x\in E$. Let $\Psi_y\from F\to \pi^{-1}(y)$  be a measure-preserving mapping for every $y\in B$ and consider the mapping
	\begin{equation*}
		\longmap{\xi}{(B\times F,\unif)}{(E,\unif)}{(y,z)}{\Psi_y(z).}
	\end{equation*}
	If $\xi$ is measurable, then it is measure preserving. Moreover, if the measures in $E$, $B$, and $F$ are normalized to have unit volume, then $\NJac\pi(x)=1$ for every $x\in E$.
\end{lemma}

\begin{proof}
	Without loss of generality, assume that the measures in $E$, $B$, and $F$ are normalized. We first check that $\NJac\pi(x)=1$ for every $x\in E$. Since $\fiberbundle{\pi}{F}{E}{B}$ is a smooth fiber bundle, we know that $\pi$ is a submersion and hence we can apply the smooth coarea formula. Therefore,
	\begin{align*}
		1&=\vol(E)=\int_{x\in E}dx=\int_{y\in B}\int_{z\in\pi^{-1}(y)}\frac{1}{C}\,dzdy=\frac{\vol(\pi^{-1}(y))\vol(B)}{C}\\
            &=\frac{\vol(F)\vol(B)}{C}=\frac{1}{C},
	\end{align*}
	and so $C=1$. Now we prove that $\xi$ is measure preserving. Let $\mathcal A\containedeq E$ be a measurable set. We have to prove that
	\begin{equation*}
		\vol(\mathcal A)=\vol(\xi^{-1}(\mathcal A)).
	\end{equation*}
	Using again the smooth coarea formula together with the fact that $\NJac\pi(x)=1$ for all $x\in E$, we have
	\begin{equation}\label{eq:vol_A}
        \begin{split}
        \vol(\mathcal A)&=\int_{x\in E}\chi_{\mathcal A}(x)\,dx=\int_{y\in B}\int_{z\in\pi^{-1}(y)}\chi_{\mathcal A}(z)\frac{1}{\NJac\pi(z)}\,dzdy\\
		&=\int_{y\in B}\int_{z\in\pi^{-1}(y)}\chi_{\mathcal A}(z)\,dzdy=\int_{y\in B}\vol(\mathcal A\intersection\pi^{-1}(y))\,dy\\
		&=\int_{y\in B}\vol(\Psi_y^{-1}(\mathcal A\intersection\pi^{-1}(y)))\,dy\\
		&=\int_{y\in B}\int_{z\in F}\chi_{\Psi_y^{-1}(\mathcal A\intersection\pi^{-1}(y))}(z)\,dzdy.
        \end{split}
	\end{equation}
	Note that
	\begin{align*}
		\chi_{\Psi_y^{-1}(\mathcal A\intersection\pi^{-1}(y))}(z)=1&\iff z\in \Psi_y^{-1}(\mathcal A\intersection\pi^{-1}(y))\iff \Psi_y(z)\in \mathcal A\intersection\pi^{-1}(y)\\
		&\iff \Psi_y(z)\in \mathcal A \iff \xi(y,z)\in \mathcal A\\
		& \iff (y,z)\in\xi^{-1}(\mathcal A)\iff \chi_{\xi^{-1}(\mathcal A)}(y,z)=1.
	\end{align*}
	Therefore, since
	\begin{align*}
		\vol(\xi^{-1}(\mathcal A))&=\int_{(y,z)\in B\times F}\chi_{\xi^{-1}(\mathcal A)}(y,z)\,d(y,z)=\int_{y\in B}\int_{z\in F}\chi_{\xi^{-1}(\mathcal A)}\,dzdy\\
		&=\int_{y\in B}\int_{z\in F}\chi_{\Psi_y^{-1}(A\intersection\pi^{-1}(y))}(z)\,dzdy\stackrel{\eqref{eq:vol_A}}{=}\vol(\mathcal A),
	\end{align*}
	the lemma follows.
\end{proof}

\begin{proof}[Proof of \cref{thm:mainresult_bundles}]
	From \cref{lemma:xi} we have that the mapping $\xi\from B\times F\to E$ given by $\xi(y,z)=\Psi_y(z)$ is measure preserving. Since $\Phi_E=\xi\composed\Phi_{B\times F}$, and both mappings are measure preserving, the theorem follows.
\end{proof}

\begin{example}[The Hopf fibration]
    Consider $\S^1\contained \C$ and $\S^{2n+1}\contained\C^{n+1}$. Recall that the (complex) Hopf fibration $\fiberbundle{h}{\S^1}{\S^{2n+1}}{\CP^n}$ is given by
    \begin{equation*}
        \longmap{h}{\S^{2n+1}}{\CP^n}{(y_1,\dotsc,y_{n+1})}{[y_1:\dotsb:y_{n+1}].}
    \end{equation*}
    The fiber of each $[y]=[y_1:\dotsb:y_{n+1}]\in\CP^n$ is a unit circle in $\S^{2n+1}$ given by
    \begin{equation*}
        h^{-1}([y])=\set{w\in\S^{2n+1}\st [w]=[y]}.
    \end{equation*}
    For each $[y]\in\CP^n$, we choose a unit norm representative $y$ smoothly out of a lower-dimensional set, and, thinking of the elements of $\S^1$ as unimodular complex numbers, we consider the mapping
    \begin{equation*}
        \longmap{\Psi_{y}}{(\S^{1},\unif)}{(h^{-1}([y]),\unif)}{\zeta}{\zeta y,}
    \end{equation*}
    which is an isometry and hence it is measure preserving. Therefore, by \cref{thm:mainresult_bundles}, the mapping
    \begin{equation*}
        \longmap{\Phi_{\S^{2n+1}}^{h}}{((0,1)^{2n+1},\unif)}{(\S^{2n+1},\unif)}{(y,t)}{\Psi_{\Phi_{\CP^n}(y)}(\Phi_{\S^1}(t))=\Phi_{\S^1}(t)\Phi_{\CP^n}(y),}
    \end{equation*}
    where $y\in\C^{n}\isomorphic \R^{2n}$ and $t\in\R$ (and recall that we are assuming that the representative of $\Phi_{\CP^n}(y)$ has unit norm) is measure preserving. Note that we have explicit expressions for both $\Phi_{\S^1}$ and $\Phi_{\CP^n}$, and hence for $\Phi_{\S^{2n+1}}^{h}$:
    \begin{align*}
        &\Phi_{\S^{2n+1}}^{h}(y,t)=\Phi_{\S^1}(t)\Phi_{\CP^n}(y)\\
        &\quad=-ie^{-i2\pi t}\frac{\Bigg(\frac{\Phi_{\R^{2n}}(y)}{\norm{\Phi_{\R^{2n}}(y)}}\bigg(-1+\frac{1}{1-\big(\frac{1}{(n-1)!}\gamma\big(n,\frac{\norm{\Phi_{\R^{2n}}(y)}^2}{2}\big)\big)^{1/n}}\bigg)^{1/2},1\Bigg)}{\norm*{\Bigg(\frac{\Phi_{\R^{2n}}(y)}{\norm{\Phi_{\R^{2n}}(y)}}\bigg(-1+\frac{1}{1-\big(\frac{1}{(n-1)!}\gamma\big(n,\frac{\norm{\Phi_{\R^{2n}}(y)}^2}{2}\big)\big)^{1/n}}\bigg)^{1/2},1\Bigg)}}.
    \end{align*}
    For the particular case of $\S^3$ we have
    \begin{align*}
      \Phi_{\S^3}^{h}(y,t)&=\Phi_{\S^1}(t)\Phi_{\CP^1}(y)=-ie^{-i2\pi t}e^{-\norm{\Phi_{\R^2}(y)}^2/4}\bigg(\frac{\Phi_{\R^2}(y)}{\norm{\Phi_{\R^2}(y)}}\sqrt{e^{\norm{\Phi_{\R^2}(y)}^2/2}-1},1\bigg)\\
    &=\bigg(-ie^{-i2\pi t}\frac{\Phi_{\R^2}(y)}{\norm{\Phi_{\R^2}(y)}}\sqrt{1-e^{-\norm{\Phi_{\R^2}(y)}^2/2}},-ie^{-i2\pi t}e^{-\norm{\Phi_{\R^2}(y)}^2/4}\bigg),
    \end{align*}
    since $\norm{\Phi_{\CP^1}(y)}=e^{\norm{\Phi_{\R^2}(y)}^2/4}$. Note that we are considering $\Phi_{\R^2}(y)\in\C$ through the canonical isomorphism $\R^2\isomorphic\C$ given by $(a,b)\mapsto a+bi$.
\end{example}

\appendix

\section{The smooth coarea formula}\label{appendix:coarea}

Let $\M,\cN$ be Riemannian manifolds. Given a smooth mapping $\varphi\from \M\to \cN$, let $D\varphi(x)\from T_x\M\to T_{\varphi(x)}\cN$ denote the differential mapping, where $T_x\M$ is the tangent space to $\M$ at $x\in \M$ and $T_{\varphi(x)}\cN$ is the tangent space to $\cN$ at $\varphi(x)\in \cN$. 

\begin{definition}[Normal Jacobian]
    Let $\M$ and $\cN$ be Riemannian manifolds and let $\varphi\from \M\to \cN$ be a $C^1$ surjective map. Let $n=\dim(\cN)$ be the real dimension of $\cN$. For every point $x\in \M$ such that the differential mapping $D\varphi(x)$ is surjective, let $v_1^x,\dotsc,v_n^x$ be an orthogonal basis of $(\ker(D\varphi(x)))^{\perp}$. Then we define the normal Jacobian of $\varphi$ at $x$, written as $\NJac\varphi(x)$, as the volume in the tangent space $T_{\varphi(x)}\cN$ of the parallelepiped spanned by $D\varphi(x)(v_1^x),\dotsc,D\varphi(x)(v_n^x)$. In the case that $D\varphi(x)$ is not surjective, we define $\NJac\varphi(x)=0$.
\end{definition}

\begin{theorem}[Smooth coarea formula]\label{thm:coarea}
    Let $\M$ and $\cN$ be two Riemannian manifolds of dimension $m$ and $n$, respectively, where $m\geq n$. Let $\varphi\from \M\to \cN$ be a smooth surjective map such that the differential mapping $D\varphi(x)$ is surjective for almost all $x\in \M$. Let $\psi\from \M\to \R$ be an integrable mapping. Then, the following equalities hold:
    \begin{align*}
        \int_{x\in \M}\psi(x)\,dx&=\int_{y\in \cN}\int_{x\in\varphi^{-1}(y)}\psi(x)\frac{1}{\NJac\varphi(x)}\,dxdy,\\
        \int_{x\in \M}\psi(x)\NJac\varphi(x)\,dx&=\int_{y\in \cN}\int_{x\in\varphi^{-1}(y)}\psi(x)\,dxdy.
    \end{align*}
\end{theorem}

Note that if $m=n$ and $\varphi$ is a diffeomorphism we recover the classical change of variables theorem.

\newpage
\section{Auxiliary computations}\label{appendix}

In this appendix, we show the explicit computations leading to the formulas in \cref{table:explicit-formulas}.

\subsection{Explicit expression of $\Phi_{\B^n}$}

Recall from \cref{prop:mainresult_ball} that we have
\begin{equation*}
    \varphi_{\B^n}(x)=\frac{x}{\norm{x}}\Bigg(\frac{\gamma\big(\frac{n}{2},\frac{\norm{x}^2}{2}\big)}{\Gamma(\frac{n}{2})} \Bigg)^{1/n}.
\end{equation*}
Hence,
\begin{equation*}
    \Phi_{\B^n}(x)=\varphi_{\B^n}(\Phi_{\R^n}(x))=\frac{\Phi_{\R^n}(x)}{\norm{\Phi_{\R^n}(x)}}\Bigg(\frac{\gamma\big(\frac{n}{2},\frac{\norm{\Phi_{\R^n}(x)}^2}{2}\big)}{\Gamma(\frac{n}{2})} \Bigg)^{1/n}.
\end{equation*}

\subsection{Explicit expression of $\Phi_{\S^1}$}

Although we could simply define $\Phi_{\S^1}(x)=e^{i2\pi x}$, let us find the expression of this mapping using the general procedure. In this case, we have 
\begin{equation*}
    \Phi_{\R}(x)=\sqrt{2}\erf^{-1}(2x-1).
\end{equation*}
Following \cref{cor:sphere}, to find $\varphi_{\S^1}$ we have to obtain $\rho$ from
\begin{equation*}
\int_0^\rho\sin^{n-1}r\,dr=\frac{\sqrt{\pi}}{\Gamma(\frac{n+1}{2})}\gamma\bigg(\frac{n}{2},\frac{r^2}{2}\bigg).
\end{equation*}
Since in this case $n=1$, we have
\begin{equation*}
    \rho(r)=\sqrt{\pi}\gamma\bigg(\frac{1}{2},\frac{r^2}{2}\bigg)=\pi\erf\bigg(\frac{r}{\sqrt{2}}\bigg).
\end{equation*}
Hence,
\begin{equation*}
    \varphi_{\S^1}(x)=\frac{\pi x}{\abs{x}}\erf\bigg(\frac{\abs{x}}{\sqrt{2}}\bigg).
\end{equation*}
Therefore,
\begin{align*}
    \varphi_{\S^1}(\Phi_{\R}(x))&=\frac{\pi \sqrt{2}\erf^{-1}(2x-1)}{\abs{\sqrt{2}\erf^{-1}(2x-1)}}\erf\bigg(\frac{\abs{\sqrt{2}\erf^{-1}(2x-1)}}{\sqrt{2}}\bigg)\\
    &=\frac{\pi \erf^{-1}(2x-1)}{\abs{\erf^{-1}(2x-1)}}\erf(\abs{\erf^{-1}(2x-1)}).
\end{align*}
Since both $\erf$ and $\erf^{-1}$ are odd functions, the absolute values cancel each other and so
\begin{equation*}
    \varphi_{\S^1}(\Phi_{\R}(x))=\pi\erf(\erf^{-1}(2x-1))=\pi(2x-1).
\end{equation*}
Recall from \cref{table:dov} that the exponential map $\exp_{\S^1}\from (-\pi,\pi)\to \S^1$ is given by
\begin{equation*}
    \exp_{\S^1}(v)=\bigg(\frac{v}{\abs{v}}\sin\abs{v},\cos\abs{v}\bigg).
\end{equation*}
Hence,
\begin{align*}
    \Phi_{\S^1}(x)&=\exp_{\S^1}(\pi(2x-1))=\bigg(\frac{\pi(2x-1)}{\abs{\pi(2x-1)}}\sin\abs{\pi(2x-1)},\cos\abs{\pi(2x-1)}\bigg)\\
    &=(\sin(2\pi x-\pi),\cos(2\pi x-\pi))=(-\sin 2\pi x,-\cos 2\pi x)\\
    &\isomorphic -ie^{-i2\pi x}.
\end{align*}

\subsection{Explicit expression of $\Phi_{\S^2}$}

Recall from \cref{cor:sphere} that we have
\begin{equation*}
    \varphi_{\S^2}(x)=\frac{x}{\norm{x}}\cdot 2\arccos e^{-\norm{x}^2/4}.
\end{equation*}
Let us compute first $\exp_{\S^2}\composed \varphi_{\S^2}$. Recall from \cref{table:dov} that
\begin{equation*}
    \exp_{\S^2}(v)=\bigg(\frac{v}{\norm{v}}\sin\norm{v},\cos\norm{v}\bigg).
\end{equation*}
Hence,
\begin{align*}
    \exp_{\S^2}(\varphi_{\S^2}(x))&=\exp_{\S^2}\bigg(\frac{x}{\norm{x}}\cdot 2\arccos e^{-\norm{x}^2/4}\bigg)\\
    &=\bigg(\frac{\frac{x}{\norm{x}}\cdot 2\arccos e^{-\norm{x}^2/4}}{\norm{\frac{x}{\norm{x}}\cdot 2\arccos e^{-\norm{x}^2/4}}}\sin\norm*{\frac{x}{\norm{x}}\cdot 2\arccos e^{-\norm{x}^2/4}},\\
    &\quad\cos\norm*{\frac{x}{\norm{x}}\cdot 2\arccos e^{-\norm{x}^2/4}}\bigg).
\end{align*}
Due to the parity of the sine and the cosine, we can rewrite the previous expression as
\begin{align*}
    \exp_{\S^2}(\varphi_{\S^2}(x))&=\bigg(\frac{x}{\norm{x}}\sin(2\arccos e^{-\norm{x}^2/4}),\cos(2\arccos e^{-\norm{x}^2/4})\bigg).
\end{align*}
To further simplfy these expressions, note that, for $-1<x<1$,
\begin{align*}
    \sin(2\arccos(x))&=2\sin(\arccos(x))\cos(\arccos(x))=2x\sqrt{1-\cos^2(\arccos(x))}\\
    &=2x\sqrt{1-x^2},
\end{align*}
and
\begin{align*}
    \cos(2\arccos(x))&=\cos^2(\arccos(x))-\sin^2(\arccos(x))=x^2-(1-x^2)\\
    &=2x^2-1.
\end{align*}
Therefore,
\begin{align*}
    \exp_{\S^2}(\varphi_{\S^2}(x))&=\bigg(\frac{x}{\norm{x}}2e^{-\norm{x}^2/4}\sqrt{1-e^{-\norm{x}^2/2}},2e^{-\norm{x}^2/2}-1\bigg).
\end{align*}
Hence, we conclude that
\begin{equation*}
    \Phi_{\S^2}(x)=\bigg(\frac{\Phi_{\R^2}(x)}{\norm{\Phi_{\R^2}(x)}}2e^{-\norm{\Phi_{\R^2}(x)}^2/4}\sqrt{1-e^{-\norm{\Phi_{\R^2}(x)}^2/2}},2e^{-\norm{\Phi_{\R^2}(x)}^2/2}-1\bigg).
\end{equation*}

\subsection{Explicit expression of $\Phi_{\RP^1}$}

In this case, we have 
\begin{equation*}
    \Phi_{\R}(x)=\sqrt{2}\erf^{-1}(2x-1).
\end{equation*}
Following \cref{cor:RPn}, to find $\varphi_{\RP^1}$ we have to obtain $\rho$ from
\begin{equation*}
\int_0^\rho\sin^{n-1}r\,dr=\frac{\sqrt{\pi}}{2\Gamma(\frac{n+1}{2})}\gamma\bigg(\frac{n}{2},\frac{r^2}{2}\bigg).
\end{equation*}
Since in this case $n=1$, we have
\begin{equation*}
    \rho(r)=\frac{\sqrt{\pi}}{2}\gamma\bigg(\frac{1}{2},\frac{r^2}{2}\bigg)=\frac{\pi}{2}\erf\bigg(\frac{r}{\sqrt{2}}\bigg).
\end{equation*}
Hence,
\begin{equation*}
    \varphi_{\RP^1}(x)=\frac{\pi x}{2\abs{x}}\erf\bigg(\frac{\abs{x}}{\sqrt{2}}\bigg).
\end{equation*}
Therefore,
\begin{align*}
    \varphi_{\RP^1}(\Phi_{\R}(x))&=\frac{\pi \sqrt{2}\erf^{-1}(2x-1)}{2\abs{\sqrt{2}\erf^{-1}(2x-1)}}\erf\bigg(\frac{\abs{\sqrt{2}\erf^{-1}(2x-1)}}{\sqrt{2}}\bigg)\\
    &=\frac{\pi \erf^{-1}(2x-1)}{2\abs{\erf^{-1}(2x-1)}}\erf(\abs{\erf^{-1}(2x-1)}).
\end{align*}
Since both $\erf$ and $\erf^{-1}$ are odd functions, the absolute values cancel each other and so
\begin{equation*}
    \varphi_{\S^1}(\Phi_{\R}(x))=\frac{\pi}{2}\erf(\erf^{-1}(2x-1))=\frac{\pi}{2}(2x-1).
\end{equation*}
Recall from \cref{table:dov} that the exponential map $\exp_{\RP^1}\from (-\pi/2,\pi/2)\to \RP^1$ is given by
\begin{equation*}
    \exp_{\RP^1}(v)=\bigg(\frac{v}{\abs{v}}\tan\abs{v},1\bigg).
\end{equation*}
Hence,
\begin{align*}
    \Phi_{\RP^1}(x)&=\exp_{\RP^1}\Big(\frac{\pi}{2}(2x-1)\Big)=\bigg(\frac{\pi(2x-1)}{\abs{\pi(2x-1)}}\tan\abs*{\frac{\pi}{2}(2x-1)},1\bigg).
\end{align*}
Since the tangent function is odd, we can simplify the previous expression as follows:
\begin{align*}
    \Phi_{\RP^1}(x)&=\bigg(\tan\Big(\frac{\pi}{2}(2x-1)\Big),1\bigg)=\bigg(\tan\Big(\pi x-\frac{\pi}{2}\Big),1\bigg)=(-\cot\pi x,1).
\end{align*}

\subsection{Explicit expression of $\Phi_{\RP^2}$}

Recall from \cref{cor:RPn} that we have
\begin{equation*}
    \varphi_{\RP^2}(x)=\frac{x}{\norm{x}}\arccos e^{-\norm{x}^2/2}.
\end{equation*}
Let us compute $\exp_{\RP^2}\composed \varphi_{\RP^2}$. Recall from \cref{table:dov} that
\begin{equation*}
    \exp_{\RP^2}(v)=\bigg(\frac{v}{\norm{v}}\tan\norm{v},1\bigg).
\end{equation*}
Hence,
\begin{equation*}
    \exp_{\RP^2}(\varphi_{\RP^2}(x))=\bigg(\frac{\frac{x}{\norm{x}}\arccos e^{-\norm{x}^2/2}}{\norm{\frac{x}{\norm{x}}\arccos e^{-\norm{x}^2/2}}}\tan\norm*{\frac{x}{\norm{x}}\arccos e^{-\norm{x}^2/2}},1\bigg).
\end{equation*}
Since the tangent function is odd, we can simplify the previous expression as follows:
\begin{equation*}
    \exp_{\RP^2}(\varphi_{\RP^2}(x))=\bigg(\frac{x}{\norm{x}}\tan\arccos e^{-\norm{x}^2/2},1\bigg).
\end{equation*}
Note that
\begin{equation*}
    \tan\arccos(x)=\frac{\sin\arccos(x)}{\cos\arccos(x)}=\frac{\sqrt{1-x^2}}{x}=\sqrt{\frac{1}{x^2}-1}.
\end{equation*}
Hence,
\begin{equation*}
    \exp_{\RP^2}(\varphi_{\RP^2}(x))=\bigg(\frac{x}{\norm{x}}\sqrt{e^{\norm{x}^2}-1},1\bigg),
\end{equation*}
and we conclude that
\begin{equation*}
    \Phi_{\RP^2}(x)=\bigg(\frac{\Phi_{\R^2}(x)}{\norm{\Phi_{\R^2}(x)}}\sqrt{e^{\norm{\Phi_{\R^2}(x)}^2}-1},1\bigg).
\end{equation*}

\subsection{Explicit expression of $\Phi_{\CP^n}$}

Recall from \cref{cor:CPn} that we have
\begin{equation*}
    \varphi_{\CP^n}(x)=\frac{x}{\norm{x}}\arcsin\Bigg(\bigg(\frac{1}{(n-1)!}\gamma\bigg(n,\frac{\norm{x}^2}{2}\bigg)\bigg)^{1/(2n)}\Bigg).
\end{equation*}
Recall from \cref{table:dov} that
\begin{equation*}
    \exp_{\CP^n}(v)=\bigg(\frac{v}{\norm{v}}\tan\norm{v},1\bigg).
\end{equation*}
Let us compute $\exp_{\CP^n}\composed\varphi_{\CP^n}$. As for the case of $\RP^2$, the parity of the tangent function implies that
\begin{equation*}
    \exp_{\CP^n}(\varphi_{\CP^n}(x))=\Bigg(\frac{x}{\norm{x}}\tan\arcsin\Bigg(\bigg(\frac{1}{(n-1)!}\gamma\bigg(n,\frac{\norm{x}^2}{2}\bigg)\bigg)^{1/(2n)}\Bigg),1\Bigg).
\end{equation*}
Note that, in our range,
\begin{equation*}
    \tan\arcsin(x)=\frac{\sin\arcsin(x)}{\cos\arcsin(x)}=\frac{x}{\sqrt{1-x^2}}=\sqrt{-1+\frac{1}{1-x^2}}.
\end{equation*}
Hence,
\begin{equation*}
    \exp_{\CP^n}(\varphi_{\CP^n}(x))=\Bigg(\frac{x}{\norm{x}}\sqrt{-1+\frac{1}{1-\big(\frac{1}{(n-1)!}\gamma(n,\frac{\norm{x}^2}{2})\big)^{1/n}}},1\Bigg),
\end{equation*}
and we conclude that
\begin{equation*}
    \Phi_{\CP^n}(x)=\Bigg(\frac{\Phi_{\R^{2n}}(x)}{\norm{\Phi_{\R^{2n}}(x)}}\sqrt{-1+\frac{1}{1-\big(\frac{1}{(n-1)!}\gamma\big(n,\frac{\norm{\Phi_{\R^{2n}}(x)}^2}{2}\big)\big)^{1/n}}},1\Bigg).
\end{equation*}


\bibliographystyle{amsplain}

\end{document}